\documentclass[11pt]{amsart}
\calclayout

\usepackage[cmtip,all]{xy}
\usepackage{color,float}
\usepackage{mathtools,comment}
\usepackage{enumitem}
\usepackage{graphicx,mathdots}
\usepackage{tikz,tikz-cd}
\usepackage{microtype}

\usepackage[colorlinks=true, linkcolor=blue, citecolor=blue, urlcolor=blue]{hyperref}
\usepackage[capitalize, nameinlink, noabbrev, sort]{cleveref}

\newcommand{\tiniest}[1]{\tiny{#1}}

\newtheorem{theorem}{Theorem}[section]
\newtheorem{lemma}[theorem]{Lemma}

\newtheorem{prop}[theorem]{Proposition}
\newtheorem{conjecture}[theorem]{Conjecture}
\crefname{section}{Section}{Sections}
\crefname{lemma}{Lemma}{Lemmas}
\crefname{theorem}{Theorem}{Theorems}

\theoremstyle{definition}
\newtheorem{defn}[theorem]{Definition}
\newtheorem{example}[theorem]{Example}
\newtheorem{question}[theorem]{Question}

\theoremstyle{remark}
\newtheorem{remark}[theorem]{Remark}

\numberwithin{equation}{section}

\newcommand{\T}{\mathfrak{t}}

\DeclareMathOperator{\PDS}{PDS}
\DeclareMathOperator{\Sym}{Sym}

\newcommand{\NN}{\mathbb{N}}

\DeclareMathOperator{\perm}{perm}

\newcommand{\Qc}{\mathcal{Q}}

\newcommand{\op}{\rho}

\DeclareMathOperator{\MLQ}{MLQ}
\DeclareMathOperator{\dg}{\texttt{dg}}
\DeclareMathOperator{\Top}{top}

\newcommand{\Z}{\mathbb{Z}}
\DeclareMathOperator{\inona}{inv-non-attacking}
\DeclareMathOperator{\qnona}{quinv-non-attacking}
\DeclareMathOperator{\isort}{inv-sorted}
\DeclareMathOperator{\qsort}{quinv-sorted}
\DeclareMathOperator{\cisort}{coinv-sorted}
\DeclareMathOperator{\cqsort}{coquinv-sorted}

\newcommand{\HH}{\widetilde{H}}

\DeclareMathOperator{\arm}{\texttt{arm}}
\DeclareMathOperator{\Arm}{Arm}
\DeclareMathOperator{\Rarm}{\widehat{Arm}}
\DeclareMathOperator{\rarm}{\widehat{\texttt{arm}}}

\DeclareMathOperator{\coquinv}{\texttt{coquinv}}

\DeclareMathOperator{\leg}{\texttt{leg}}
\DeclareMathOperator{\maj}{\texttt{maj}}

\DeclareMathOperator{\inv}{\texttt{inv}}
\DeclareMathOperator{\inc}{inc}
\DeclareMathOperator{\dec}{dec}

\DeclareMathOperator{\prob}{prob}
\DeclareMathOperator{\sort}{sort}

\DeclareMathOperator{\coinv}{\texttt{coinv}}

\DeclareMathOperator{\wt}{wt}

\DeclareMathOperator{\quinv}{\texttt{quinv}}
\DeclareMathOperator{\Bot}{bot}

\DeclareMathOperator{\HHL}{HHL}
\DeclareMathOperator{\PR}{PR}

\DeclareMathOperator{\Tab}{Tab}
\DeclareMathOperator{\South}{South}

\newcommand{\qbinom}[2]{\bgroup\renewcommand*{\arraystretch}{1}\begin{bmatrix} #1 \\ #2\end{bmatrix} \egroup}

\newlength\cellsize 
\savebox2{%
\begin{picture}(12,12)
\put(0,0){\line(1,0){12}}
\put(0,0){\line(0,1){12}}
\put(12,0){\line(0,1){12}}
\put(0,12){\line(1,0){12}}
\end{picture}}

\newcommand\cellify[1]{\def\thearg{#1}\def\nothing{}%
\ifx\thearg\nothing
\vrule width0pt height\cellsize depth0pt\else
\hbox to 0pt{\usebox2\hss}\fi%
\vbox to 12\unitlength{
\vss
\hbox to 12\unitlength{\hss$#1$\hss}
\vss}}

\newcommand\tableau[1]{\vtop{\let\\=\cr
\setlength\baselineskip{-16000pt}
\setlength\lineskiplimit{16000pt}
\setlength\lineskip{0pt}
\halign{&\cellify{##}\cr#1\crcr}}}

\savebox3{%
\begin{picture}(12,12)
\put(0,0){\line(1,0){12}}
\put(0,0){\line(0,1){12}}
\put(12,0){\line(0,1){12}}
\put(0,12){\line(1,0){12}}
\end{picture}}

\newcommand\expath[1]{%
\hbox to 0pt{\usebox3\hss}%
\vbox to 12\unitlength{
\vss
\hbox to 12\unitlength{\hss$#1$\hss}
\vss}}

\newcommand\cell[3]{
\def\i{#1} \def\j{#2} \def\entry{#3}

\draw (\j-1,-\i)--(\j,-\i)--(\j,-\i+1)--(\j-1,-\i+1)--(\j-1,-\i);
\node at (\j-.5,-\i+.5) {\entry};
}

\newcommand{\qtrip}[3]{
\begin{tikzpicture}[scale=0.5]
\cell{1}{0}{#1} \cell{2}{0}{#2} \cell{2}{2.7}{#3}
\node at (1,-1.5) {$\cdots$};
\end{tikzpicture}
}

\newcommand{\HHLtrip}[3]{
\begin{tikzpicture}[scale=0.5]
\cell{1}{0}{#1} \cell{2}{0}{#2} \cell{1}{2.7}{#3}
\node at (1,-.5) {$\cdots$};
\end{tikzpicture}
}

\newcommand\circleS[4]{
\def \n {4}
\def \radius {.5cm}
\def \margin {20}
\edef\s{0}
 \pgfmathparse{\s+1};
\foreach \en in {#1,#2,#3,#4} {%
  \node at ({360/\n * (\s - 1)}:\radius) {$\en$};
   \draw[>=latex] ({360/\n * (\s - 1)+\margin}:\radius) 
     arc ({360/\n * (\s - 1)+\margin}:{360/\n * (\s)-\margin}:\radius);
  \pgfmathparse{\s+1}
      \xdef\s{\pgfmathresult}
}
}

\newcommand\circleU[3]{
\def \n {3}
\def \radius {.5cm}
\def \margin {20}
\edef\s{0}
 \pgfmathparse{\s+1};
\foreach \en in {#1,#2,#3} {%
  \node at ({360/\n * (\s - 1)}:\radius) {$\en$};
   \draw[>=latex] ({360/\n * (\s - 1)+\margin}:\radius) 
     arc ({360/\n * (\s - 1)+\margin}:{360/\n * (\s)-\margin}:\radius);
  \pgfmathparse{\s+1}
      \xdef\s{\pgfmathresult}
}
}

\newcommand{\Ac}{\mathcal{A}}

\newcommand{\I}{\mathfrak{S}}

\begin{document}

\title[A compact formula for the symmetric Macdonald polynomials]{Probabilistic operators for non-attacking tableaux and a compact formula for the symmetric Macdonald polynomials}

\author{Olya Mandelshtam}
\address{University of Waterloo, Waterloo, Ontario, Canada}
\email{omandels@uwaterloo.ca}

\date{May 25, 2025}

\begin{abstract}

 We prove a new tableaux formula for the symmetric Macdonald polynomials $P_{\lambda}(X;q,t)$ that has considerably fewer terms and simpler weights than previously existing formulas. Our formula is a sum over certain sorted non-attacking tableaux, weighted by the \emph{queue inversion} statistic $\quinv$. The $\quinv$ statistic originates from a formula for the modified Macdonald polynomials $\widetilde{H}_{\lambda}(X;q,t)$ due to Ayyer, Martin, and the author (2022), and is naturally related to the dynamics of the asymmetric simple exclusion process (ASEP) on a circle. 

We prove our results by introducing probabilistic operators that act on non-attacking tableaux to generate a set of tableaux whose weighted sum equals $P_{\lambda}(X;q,t)$. These operators are a modification of the \emph{inversion flip operators} of Loehr and Niese (2012), which yield an involution on tableaux that preserves the major index statistic, but fails to preserve the non-attacking condition. Our tableaux are in bijection with the multiline queues introduced by Martin (2020), allowing us to derive an alternative multiline queue formula for $P_{\lambda}(X;q,t)$. Finally, our formula recovers an alternative formula for the Jack polynomials $J_{\lambda}(X;\alpha)$ due to Knop and Sahi (1996) using the same queue inversion statistic.

\end{abstract}

\maketitle

\section{Introduction}
 Let $X=\{x_1,x_2,\cdots\}$ be an infinite set of indeterminates, and for $n>0$, let $X_n=\{x_1,\ldots,x_n\}$ denote the truncation to the first $n$ variables. The \emph{symmetric Macdonald polynomials} $P_{\lambda}(X;q,t)$, introduced by Macdonald \cite{Mac88}, form a family of multivariate orthogonal polynomials in $X$, indexed by partitions, with coefficients that are rational functions of $q$ and $t$. The \emph{modified Macdonald polynomials} $\HH_{\lambda}(X;q,t)$ were introduced by Garcia and Haiman \cite{GarsiaHaiman96} as a transformed version of the polynomials $P_{\lambda}$ with coefficients in $\NN[q,t]$. The exploration of combinatorial formulas for the $P_{\lambda}$'s, the $\HH_{\lambda}$'s, and associated functions has been an active area of study since their inception, and has led to various formulas for these polynomials through a variety of combinatorial objects. Most famously, in 2004, an elegant tableaux formula for $\HH_{\lambda}$ in terms of a certain \emph{inversion statistic} $\inv$, was conjectured by Haglund, and subsequently proved by Haglund, Haiman, and Loehr \cite{HHL05}. This has come to be known as the HHL formula: \begin{equation}\label{eq:H inv}
\HH_{\lambda}(X;q,t)=\sum_{\sigma:\dg(\lambda)\rightarrow \Z^+}x^{\sigma}q^{\maj(\sigma)}t^{\inv(\sigma)}.
\end{equation} 
Using a technique called \emph{superization}, \eqref{eq:H inv} leads to a formula \eqref{eq:P inv} for $P_{\lambda}$ with the (co)inversion statistic on certain \emph{non-attacking fillings}. However, this formula has two drawbacks: it has a cumbersome pre-factor, and many terms in the sum collapse to a single term in the polynomial, making the computation highly inefficient when $\lambda$ is not a strict partition. 

Recent studies exploring the role of Macdonald polynomials in statistical mechanics models have led to several new combinatorial formulas. In \cite{CGW-2015}, motivated by connections with the \emph{asymmetric simple exclusion process (ASEP)}, Cantini, de Gier, and Wheeler found that 
\[P_{\lambda}(X_n;q,t)=\sum_{\alpha} f_{\alpha}(X_n;q,t),
\]
where the sum is over all $\alpha$'s of length $n$ that sort to $\lambda$, and the $f_{\alpha}$'s are known as the \emph{ASEP polynomials}. These polynomials specialize to the stationary probabilities of the ASEP, with $P_{\lambda}(X_n;q,t)$ reducing to the partition function of the ASEP of type $\lambda$ on $n$ sites when $x_1=\cdots=x_n=q=1$. In \cite{CMW18}, it was shown by Corteel, Williams, and the author that the functions $f_{\alpha}$ coincide with certain permuted basement Macdonald polynomials studied in \cite{Ale16, Fer11} that generalize the nonsymmetric Macdonald polynomials. Building upon Martin's \emph{multiline queue} formula for computing the stationary probabilities of the ASEP \cite{martin-2020},  \cite{CMW18} described a related formula for these polynomials in terms of (slightly modified) multiline queues and a \emph{queue inversion statistic}. This led to the first ``efficient'' formula for $P_{\lambda}(X;q,t)$ using statistics with a probabilistic interpretation derived from the ASEP dynamics. However, multiline queues cannot be naturally represented as tableaux in a way that preserves both their statistics and their direct connection to modified Macdonald polynomials \cite[Section 5]{CMW18}. 

On the other hand, expressing $P_{\lambda}$ as a sum over certain nonsymmetric Macdonald polynomials yielded an efficient tableaux formula in terms of the generalized HHL statistics defined in \cite{HHL08} for composition shapes. However, those generalized HHL statistics do not naturally correspond to multiline queue statistics, thereby losing the explicit connection with ASEP dynamics. Furthermore, it is not currently known how to use these generalized statistics to recover the plethystic relationship between $P_{\lambda}$ and $\HH_{\lambda}$.

The interpretation of the multiline queue statistics on tableaux inspired the authors of \cite{CHMMW20} to conjecture a new tableaux formula for $\HH_{\lambda}$ using the queue inversion statistic $\quinv$:
\begin{equation}\label{eq:H quinv}
\HH_{\lambda}(X;q,t)=\sum_{\sigma\in\dg(\lambda)}x^{\sigma}q^{\maj(\sigma)}t^{\quinv(\sigma)}.
\end{equation}
This formula was subsequently proved by Ayyer, Martin, and the author in \cite{AMM20}. Furthermore, these authors discovered that the modified Macdonald polynomial similarly decomposes into nonsymmetric components which specialize at $q=1$ to the stationary probabilities of the \emph{multispecies totally asymmetric zero range process (TAZRP)} \cite{AMM22}, with $\HH_{\lambda}(X_n;q,t)$ specializing at $q=1$ to the partition function of the TAZRP of type $\lambda$ on $n$ sites. 

In this article, we prove the following compact tableaux formula for $P_{\lambda}$ in terms of the $\quinv$ statistic on $\coquinv$-\emph{sorted non-attacking tableaux}, which are defined in \cref{sec:background}. This formula results from compressing the formula \eqref{eq:P quinv} obtained in \cite{Man23}, and was conjectured therein. 
\begin{theorem}\label{thm:main}
The symmetric Macdonald polynomial $P_{\lambda}(X;q,t)$ is given by
\begin{equation}\label{eq:main}
P_{\lambda}(X;q,t) =  \sum_{\substack{\sigma:\dg(\lambda)\rightarrow\Z^+\\\sigma\,\qnona\\
\sigma\,\cqsort}}q^{\maj(\sigma)}t^{\coquinv(\sigma)}x^{\sigma}\prod_{\substack{u\in\dg(\lambda)\\\sigma(u)\neq\sigma(\South(u))}}\frac{1-t}{1-q^{\leg(u)+1}t^{\rarm(u)+1}}.
\end{equation}
\end{theorem}

We show that the $\coquinv$-sorted tableaux with the statistic $\quinv$ in \eqref{eq:main} are in bijection with the multiline queues of Martin \cite{martin-2020}. As a consequence of this bijection, we obtain an alternative multiline queue formula for $P_{\lambda}$ through Martin's multiline queues. Moreover, our formula can be split into nonsymmetric components which specialize at $x_1=\cdots=x_n=q=1$ to probabilities of the ASEP, and which we conjecture are the ASEP polynomials. 

Our objective in this paper is to show that the $\quinv$ and $\widehat{\texttt{arm}}$ statistics in \eqref{eq:main} are natural statistics for a tableaux formula for $P_{\lambda}$, as they exhibit the relationship of the classical and modified Macdonald polynomials to the ASEP and TAZRP, respectively. Moreover, the two formulas are directly connected through a combinatorial interpretation of the plethystic substitution that transforms $P_{\lambda}$ into $\HH_{\lambda}$. These relationships are shown in the diagram below.
\medskip

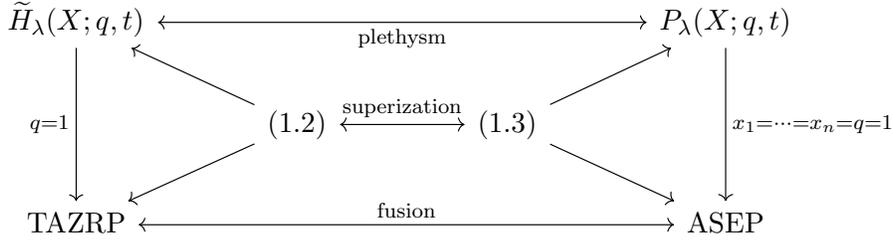
\hspace{1in}
\begin{tikzcd}
     \HH_{\lambda}(X;q,t) \arrow[<-,rrr,"\text{plethystic substitution}"] \arrow[<-,dr] &\qquad\qquad&\qquad\qquad&P_{\lambda}(X;q,t)\arrow[<-,dl]\arrow[ddd, "X=q=1"]\\
& \eqref{eq:H quinv}    \arrow[->,r,"\text{superization}"] \arrow[->,ddl,"q=1"] & \eqref{eq:P quinv}\arrow[d,"\text{compression}"]\\
&& \eqref{eq:main}\arrow[dr]\\
     \text{TAZRP}\arrow[<-,uuu, "q=1"] &&&\text{ASEP}\arrow[<-,ul,"X=q=1"] 
\end{tikzcd}
\medskip

This article is organized as follows. \cref{sec:background} contains the necessary background on formulas for $\HH_{\lambda}(X;q,t)$ and $P_{\lambda}(X;q;t)$ in terms of the $\inv$ and $\quinv$ statistics from \cite{HHL05,HHL08} and \cite{AMM20,Man23}, respectively. In \cref{sec:operators}, we define the probabilistic entry-swapping operators on $\quinv$-non-attacking fillings which we use to prove \cref{thm:main}, with a $\coinv$ analog in \cref{sec:inv operators}. In \cref{sec:martin}, we discusses the map to Martin's multiline queues in \cref{sec:martin}, where we also conjecture a new tableaux formula for the ASEP polynomials $f_\alpha(X_n;q,t)$. Finally, we derive an alternative formula for Jack polynomials using the $\quinv$ statistic in \cref{sec:jack}.

\section{Background}\label{sec:background}

\subsection{Combinatorial statistics and tableaux}
A \emph{composition} is a sequence $\alpha=(\alpha_1,\dots, \alpha_k)$, of $k$ non-negative integers, called its \emph{parts}. The number of nonzero parts is given by $\ell(\alpha)$, and $|\alpha|=\sum_i\alpha_i$ is the sum of the parts. Denote by $\sort(\alpha)$ and $\inc(\alpha)$ the compositions obtained by rearranging the parts of $\alpha$ in weakly decreasing and increasing order, respectively. If $\alpha$ is a weakly decreasing composition, we call it a \emph{partition}. For a partition $\lambda$, define the conjugate partition $\lambda'=(\lambda_1',\ldots,\lambda_{\ell(\lambda)}')$ by $\lambda'_j=|\{i:\lambda_i\geq j\}|$, and define $n(\lambda)=\sum_i {\lambda_i'\choose 2}$. We may equivalently use the \emph{frequency notation} to describe a partition $\lambda$ as a multiplicity vector. Let $m_i(\lambda)$ be the number of parts of size $i$ in $\lambda$; then we may write $\lambda$ as $\langle 1^{m_1(\lambda)}2^{m_2(\lambda)}\ldots\rangle$. If each part of $\lambda$ appears at most once, we call $\lambda$ a \emph{strict partition}.

\begin{example}
   For example, the composition $\alpha=(3,1,4,3,3,1)$ rearranges to $\inc(\alpha)=(1,1,3,3,3,4)$ and $\lambda=\sort(\alpha)=(4,3,3,3,1,1)$. Then $\lambda'=(6,4,4,1)$ with $n(\lambda)=27$, and $\lambda$ has frequency notation $\langle 1^22^03^34^1\rangle$.
\end{example}

Given a partition or composition $\alpha=(\alpha_1,\dots, \alpha_n)$, its \emph{diagram} $\dg(\alpha)$ is a sequence of bottom justified columns with $\alpha_i$ cells in the $i$'th column. In this article we will only be considering diagrams for partitions (although the leftmost and rightmost diagrams in \cref{fig:leg} are composition diagrams). Note that in this article, $\dg(\lambda)$ corresponds to the Ferrers diagram of the conjugate partition $\lambda'$.
  
The columns of a partition or composition diagram $\dg(\alpha)$ are labeled from left to right, and the rows from bottom to top. The notation $u=(r,c)$ refers to the cell in row $r$ and column $c$ (opposite of the convention of Cartesian coordinates). We denote by $\South(u)=\South(r,c)$ the cell $(r-1,c)$ directly below $u$, if it exists; if $r=1$ and $\South(u)$ doesn't exist, we set the convention $\South(u)=\infty$. The \emph{leg} of a cell $u=(r,j)\in\dg(\alpha)$ is denoted $\leg(u)=\alpha_j-r$ and is equal to the number of cells above in the same column as $u$. See \cref{fig:leg} for an example.

\begin{figure}[h]
  \centering
\begin{tikzpicture}[scale=.5]
\node at (0,0) {\tableau{&&\ \\&\ &\ &&\ \\\ &\ &u&&\ &\ \\\ &\ &\ &\ &\ &\ }};
\node at (8,0) {\tableau{\ \\\ &\ &\ \\\ &\ &u&\ &\ \\\ &\ &\ &\ &\ &\ }};
\node at (15,0) {\tableau{&&&&&\ \\&&&\ &\ &\ \\&\ &u&\ &\ &\ \\\ &\ &\ &\ &\ &\ }};

\end{tikzpicture} 
\caption{We show the diagrams of the compositions $\alpha=(2,3,4,1,3,2)$, $\lambda(\alpha)$, and $\inc(\alpha)$. The cell $u=(2,3)$ has $\leg(u)$ equal to $2, 1, 0$ in each diagram, respectively. }
\label{fig:leg}
\end{figure}
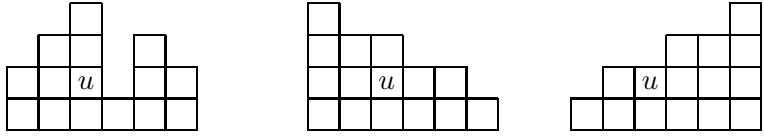

A \emph{filling} $\sigma:\dg(\alpha) \to \Z^+$ is an assignment of positive integers to the cells of $\dg(\alpha)$. For a cell $u\in\dg(\alpha)$, $\sigma(u)$ denotes the entry in the cell $u$. The content of a filling $\sigma$ is recorded as a monomial in the variables $X=x_1,x_2,\ldots$, and is denoted by $x^{\sigma} := \prod_{u\in \dg(\lambda)} x_{\sigma(u)}$. Define the \emph{major index} of a filling $\sigma$, denoted $\maj(\sigma)$, by
\[
\maj(\sigma) = \sum_{\substack{u\in \dg(\lambda)\\\sigma(u)>\sigma(\South(u))}} (\leg(u)+1).
\]

We shall also define two types of \emph{non-attacking fillings}.
\begin{defn}\label{def:nonattacking}
For two cells $x=(r_1,i),y=(r_2,j)\in\dg(\lambda)$ with $i<j$, we say they are $\inv$-\emph{attacking} if $r_1=r_2$ or $r_2=r_1+1$. We say $x$ and $y$ are $\quinv$-\emph{attacking} if $r_1=r_2$ or if $i<j$ and $r_2=r_1-1$. The configurations comparing the two definitions are shown below: 
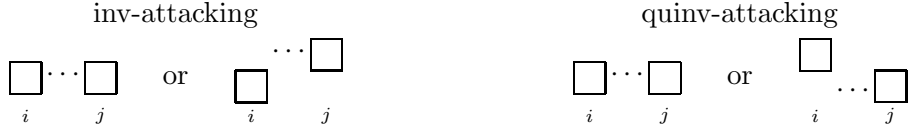
\begin{center}
\begin{tikzpicture}[scale=0.5]
\def\h{15};
\node at (0,1.7) {$\inv$-attacking};
\node at (-3,0) {$\raisebox{-5pt}{\tableau{\ }}\cdots\raisebox{-5pt}{\tableau{\ }}$};
\node at (0,0) {or};
\node at (3,.2) {$\raisebox{-5pt}{\tableau{\\\ }}\cdots\raisebox{-5pt}{\tableau{\ \\}}$};
\node at (-4,-1) {\tiny$i$};
\node at (-2,-1) {\tiny$j$};
\node at (2,-1) {\tiny$i$};
\node at (4,-1) {\tiny$j$};
\node at (\h,1.7) {$\quinv$-attacking};
\node at (-3+\h,0) {$\raisebox{-5pt}{\tableau{\ }}\cdots\raisebox{-5pt}{\tableau{\ }}$};
\node at (0+\h,0) {or};
\node at (3+\h,.2) {$\raisebox{10pt}{\tableau{\ \\}}\cdots\raisebox{10pt}{\tableau{\\\ }}$};
\node at (-4+\h,-1) {\tiny$i$};
\node at (-2+\h,-1) {\tiny$j$};
\node at (2+\h,-1) {\tiny$i$};
\node at (4+\h,-1) {\tiny$j$};
\end{tikzpicture}
\end{center}
A filling of $\dg(\lambda)$ is $\inv$-\emph{non-attacking} if no pair of $\inv$-attacking cells contains the same entry.  Similarly, a filling of $\dg(\lambda)$ is $\quinv$-\emph{non-attacking} if no pair of $\quinv$-attacking cells contains the same entry. 
\end{defn}

\begin{remark}
For a filling of partition shape, the $\quinv$-non-attacking condition is more restrictive than the $\inv$-non-attacking condition (there are more $\quinv$-attacking pairs). Thus a formula over $\quinv$-non-attacking tableaux will have fewer terms than one over $\inv$-non-attacking tableaux.
\end{remark}

\subsection{Definitions of $P_{\lambda}$, $J_{\lambda}$, and  $\HH_{\lambda}$}
For $X=x_1,x_2,\ldots$ a set of indeterminates, let $\Sym$ be the ring of symmetric functions in $X$ with coefficients in $\mathbb{Q}(q,t)$. The Macdonald polynomial $P_{\lambda}(X;q,t)$ is characterized as the unique polynomial in $\Sym$ that is orthogonal with respect to the $q,t$-generalization of the Hall inner product, and can be written as
\[
P_{\lambda}(X;q,t)=m_{\lambda}(X)+\sum_{\mu< \lambda} b_{\mu,\lambda}(q,t)m_{\mu}(X),
\]
where the usual dominance order on partitions is used: $\mu\leq \lambda$ if and only if $\mu_1+\cdots+\mu_k\leq \lambda_1+\cdots+\lambda_k$ for all $k\geq 1$. When $P_{\lambda}(X;q,t)$ is scaled by a certain product of linear terms, one obtains the \emph{integral form} $J_{\lambda}(X;q,t)$, whose coefficients are in $\Z[q,t]$. Define
\begin{equation}
\PR_{\lambda}(q,t)=\prod_{u\in\dg(\lambda)}(1-q^{\leg(u)}t^{\arm(u)+1}),\qquad \widetilde{\PR}_{\lambda}(q,t)=\prod_{u\in\overline{\dg}(\lambda)}(1-q^{\leg(u)+1}t^{\arm(u)+1})
\end{equation}
where $\overline{\dg}(\lambda)$ is the set of cells in $\dg(\lambda)$ not in the bottom row. The integral form is then defined as $J_{\lambda}(X;q,t)=\PR_{\lambda}(q,t) P_{\lambda}(X;q,t)$.
 
In \cite{GarsiaHaiman96}, Garcia and Haiman defined the modified Macdonald polynomials by plethystic substitution into $J_{\lambda}(X;q,t)$: 
\begin{equation}\label{eq:HJ}
\HH_{\lambda}(X;q,t)=t^{n(\lambda)}J_{\lambda}\Big[\frac{X}{1-t^{-1}};q,t^{-1}\Big].
\end{equation}
We postpone the definition of plethystic substitution until \cref{sec:plethysm}, where the plethystic relationship in \cref{eq:HJ} is interpreted combinatorially through \emph{superization}. This interpretation was used in \cite{HHL05} to derive the following formula for $P_{\lambda}$ in terms of $\inv$-non-attacking fillings: 
\begin{equation}\label{eq:P inv}
P_{\lambda}(X;q,t)=\frac{\widetilde{\PR}_{\lambda}(q,t)}{\PR_{\lambda}(q,t)}\hspace{-0.2in}\sum_{\substack{\sigma:\dg(\lambda)\rightarrow\Z^+\\\sigma\,\inona}}x^{\sigma}q^{\maj(\sigma)}t^{\coinv(\sigma)}\prod_{\substack{u\in\dg(\lambda)\\\sigma(u)\neq\sigma(\South(u))}}\frac{1-t}{1-q^{\leg(u)+1}t^{\arm(u)+1}}. 
\end{equation}
For a discussion of the pre-factor $\widetilde{\PR}_{\lambda}(q,t)/\PR_{\lambda}(q,t)$, see \cite[Section 4]{CHMMW20}.

Although the statistics $\maj$ and $\coinv$ appearing in the above formula are well-behaved, the formula itself is unsatisfying because many of the tableaux appearing in the sum compress to a single term in the polynomial, making the computation quite inefficient for certain $\lambda$'s. For example, when $\lambda=(1^k)$, $P_{\lambda}=m_{1^k}$, and the coefficient $[x^{\lambda}]P_{\lambda}(X;q,t)=1$, whereas in the formula above, there are $k!$ tableaux contributing to the coefficient of $x^{\lambda}$.

The multiline queue formula of \cite{CMW18} was a significant improvement over \eqref{eq:P inv} in the sense of having fewer terms; however, translating it into the tableaux setting resulted in having to define complicated statistics, making such a tableaux formula impractical. On the other hand, in \cite[Equation 4.10]{CHMMW20}, a tableaux formula was given in terms of sorted non-attacking fillings of $\dg(\inc(\lambda))$ and the generalized statistics $\coinv'$ and $\arm'$ from \cite{HHL05}. This formula was based on the fact that $P_{\lambda}$ can be written as a sum over certain \emph{permuted basement Macdonald polynomials}. However, the statistics $\coinv'$ and $\arm'$ are more complicated than the statistics $\arm$ and $\inv$ in \eqref{eq:P inv}, and don't correspond directly to statistics on multiline queues.

\subsection{Two formulas for $\HH_{\lambda}$}
In this section, we describe the formulas \eqref{eq:H inv} due to \cite{HHL05} and \eqref{eq:H quinv} due to \cite{AMM20} for $\HH_{\lambda}(X;q,t)$.

Define the function $\Qc:\NN^3\rightarrow \{0,1\}$ by $\Qc(a,b,c)=0$ if the entries $a,b,c$ satisfy one of the following inequalities and $\Qc(a,b,c)=1$ otherwise: 
\begin{equation}\label{eq:ineq}
a\leq b<c\quad\text{or}\quad c<a\leq b \quad\text{or}\quad b<c<a.
\end{equation}

\subsubsection{HHL statistics} For a cell $u=(r,k)\in\dg(\lambda)$, its \emph{arm} denoted $\Arm(u)$ is the set of cells in its row to its right, and $\arm(u)=|\Arm(u)|=\lambda'_r-k$.  Below, the shaded cells correspond to $\Arm(u)$:
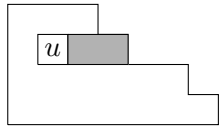
\[
\scalebox{1}{\begin{tikzpicture}[scale=.4]
\draw (-1,0)--(6,0)--(6,1)--(5,1)--(5,2)--(3,2)--(3,3)--(2,3)--(2,4)--(-1,4)--(-1,0);
\draw (0,2)--(0,3)--(1,3)--(1,2)--(0,2);
\node at (.5,2.5) {$u$};
\filldraw[black, fill=black!30] (1,3)--(3,3)--(3,2)--(1,2)--(1,3);
\end{tikzpicture}}
\]
A $\Gamma$-triple is a triple of cells $x=(r,i), y=(r-1,i), z=(r,j)$ where $j>i$, so that $z\in\Arm(x)$. When $r=1$, it is called a \emph{degenerate $\Gamma$-triple}. A $\Gamma$-triple is an \emph{inversion} in a filling $\sigma$ of $\dg(\lambda)$ if the entries $a=\sigma(x)$, $b=\sigma(y)$, $c=\sigma(z)$ in the configuration below are cyclically increasing when read counterclockwise. Ties are broken with respect to top-to-bottom, left-to-right reading order on the entries. By convention, $\sigma(y)=\infty$ in a degenerate triple:
\begin{center}
\raisebox{-10pt}{\HHLtrip{$a$}{$b$}{$c$}} \qquad or \qquad \raisebox{-10pt}{\begin{tikzpicture}[scale=0.5]
\node at (-0.5,-1.5) {$\infty$};
 \cell{1}{0}{$a$} \cell{1}{2.7}{$c$}
\node at (1,-0.5) {$\cdots$};
\end{tikzpicture}} \quad (degenerate).
\end{center}
Equivalently, a $\Gamma$-triple with contents $a,b,c$ in the configuration above is an inversion if $\Qc(a,b,c)=0$, and a \emph{coinversion} if $\Qc(a,b,c)=1$. The total number of inversion triples in $\sigma$ is denoted by $\inv(\sigma)$, and the number of coinversion triples is $\coinv(\sigma)=n(\lambda)-\inv(\sigma)$. 

\subsubsection{Queue statistics} For a cell $u=(r,k)\in\dg(\lambda)$, its arm $\Rarm(u)$ is the set of cells in the \emph{row below} strictly to its right, and $\rarm(u)=|\Rarm(u)|=\lambda'_{r-1}-k$.  An $L$-triple is a triple of cells $x=(r,i)$, $y=(r-1,i)$, $z=(r-1,j)$ where $j>i$, so that $z\in\Rarm(x)$. Below, the shaded cells correspond to $\Rarm(u)$:
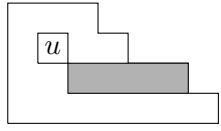
\[
\scalebox{1}{\begin{tikzpicture}[scale=.4]
\draw (-1,0)--(6,0)--(6,1)--(5,1)--(5,2)--(3,2)--(3,3)--(2,3)--(2,4)--(-1,4)--(-1,0);
\draw (0,2)--(0,3)--(1,3)--(1,2)--(0,2);
\node at (.5,2.5) {$u$};
\filldraw[black, fill=black!30] (1,2)--(5,2)--(5,1)--(1,1)--(1,2);
\end{tikzpicture}}
\]
When $r=\lambda_k$, it is called a \emph{degenerate $L$-triple}. An $L$-triple is a \emph{queue inversion} (or $\quinv$) in a filling $\sigma$ if the entries $a=\sigma(x)$, $b=\sigma(y)$, $c=\sigma(z)$ in the configuration below are cyclically increasing when read counterclockwise. Ties are broken with respect to top-to-bottom, \emph{right-to-left} reading order on the entries. By convention $\sigma(x)=0$ in a degenerate triple: 
\begin{center}
\raisebox{-10pt}{\qtrip{$a$}{$b$}{$c$}}, \qquad or \qquad \raisebox{-10pt}{\begin{tikzpicture}[scale=0.5]
\node at (-0.5,-.5) {$0$};
 \cell{2}{0}{$b$} \cell{2}{2.7}{$c$}
\node at (1,-1.5) {$\cdots$};
\end{tikzpicture}} \quad (degenerate).
\end{center}
Equivalently, an $L$-triple with contents $a,b,c$ in the configuration above is a $\quinv$ triple if $\Qc(a,b,c)=0$, and if $\Qc(a,b,c)=1$ we call it a $\coquinv$ triple. The total number of $\quinv$ triples in $\sigma$ is denoted by $\quinv(\sigma)$, and the number of $\coquinv$ triples is $\coquinv(\sigma)=n(\lambda)-\quinv(\sigma)$. 

With these statistics, $\HH_{\lambda}(X;q,t)$ can be written as a sum over fillings $\sigma:\dg(\lambda)\rightarrow\Z^+$ either in terms of the $\inv$ statistic in \eqref{eq:H inv} or in terms of the $\quinv$ statistic in \eqref{eq:H quinv}. 

\begin{remark}
Recently, an interesting bijection relating the statistics $\inv$ and $\quinv$ was found in \cite{JL24}. 
\end{remark}

\begin{example}\label{ex:stats}
In the tableau $\sigma$ below, the cell $u=(2,2)$ has $\leg(u)=1$, $\arm(u)=3$ and $\rarm(u)=7$. $\maj(\sigma)=5$, $\inv(\sigma)=6$, and $\quinv(\sigma)=14$. Thus this filling contributes $x_1^5x_2^3x_3^6x_4x_5q^5t^{6}$ to \eqref{eq:H inv} and $x_1^5x_2^3x_3^6x_4x_5q^5t^{14}$ to \eqref{eq:H quinv}. 
\[
\sigma=\raisebox{10pt}{\tableau{3&2\\1&3&3&1&3\\1&1&2&1&2&4&4&3&3}}
\]
\end{example}

\subsection{Formulas for $P_{\lambda}$ and $J_{\lambda}$ from $\HH_{\lambda}$ via superization}\label{sec:plethysm}
We recall briefly the definition of plethystic substitution (see \cite[Chapter 1]{Hag08} for a full treatment). Let $X=x_1,x_2,\ldots$ and $Y=y_1,y_2,\ldots$ be two alphabets. For a formal power series or polynomial $A\in\Sym$ in the indeterminates $X$, define the \emph{plethystic substitution} of $A$ into the power sum symmetric function $p_k=\sum_j x_j^k$ by $p_k[A]=A(x_1^k,x_2^k,\ldots)$. For a symmetric function $f$, extend the definition of plethystic substitution by expanding $f$ in the power sum basis and substituting $p_k[A]$ for $p_k(A)$ to obtain $f[A]$.   By convention, when $X$ appears inside brackets, it is treated as a formal power series as $f[X]=f[x_1+x_2+\cdots]$, so that $f[X]=f(X)$ and $f[X+Y]=f(X,Y)$. We will also use the notation $f[\varepsilon X]=f(-X)=f(-x_1,-x_2,\ldots)$.

We recall the classical notion of \emph{superization} of a symmetric function $f\in\Sym$ (see, for instance, \cite{HHLRU} for details). Define the \emph{super-alphabet} $\Ac=\Z^+\cup\Z^-=\{1,\bar1,2,\bar2,\ldots\}$, where the elements of the usual alphabet $\Z^+=\{1,2,\ldots\}$ are referred to as the ``positive'' letters, and those of $\Z^-=\{\bar1,\bar2,\ldots\}$ as the ``negative'' letters. The total ordering on $\Ac\cup \{0,\infty\}$ is chosen to be:
\[
0<1<\bar1<2<\bar2<\cdots<\infty.
\]
When standardizing a word in the alphabet $\Ac$, ties between two equal letters are broken by saying the leftmost one is smaller if they are positive, and the leftmost one is larger if they are negative. To make this precise, we introduce the notation $I(a,b)$ for any $a,b\in\Ac\cup \{0,\infty\}$ to generalize the notion of a descent on $\Ac\cup \{0,\infty\}$ with respect to the total ordering above:
\[
I(a,b)=\begin{cases} 1,&a>b\ \text{or}\ a=b\in\Z^-,\\
0,&a<b\ \text{or}\ a=b\in\Z^+.
\end{cases}
\]
Note in particular that if $a\neq b$, then $I(a,b)=\delta_{a>b}$ (where $\delta_F$ is the truth function that outputs 1 if $F$ is true and 0 otherwise). Otherwise, we have $I(a,a)=0$ and $I(\bar a,\bar a)=1$ for any $a\in \Z^+$. For example, $I(1,2)=I(\bar 1,\bar 2)=I(2,\bar 2)=I(1,1)=I(1,\infty)=0$ and $I(2,1)=I(\bar 2,2)=I(\bar 2,\bar 2)=I(1,0)=1$.

\begin{defn}
A \emph{super filling} of a diagram $\dg(\lambda)$ with the super alphabet $\Ac$ with a fixed total ordering is a function $\sigma: \dg(\lambda)\rightarrow \Ac$. We denote by $|\sigma|$ the regular filling with the alphabet $\Z^+$ obtained by sending $\bar a$ in $\sigma$ to $a$ for each $a\in\Z^+$. For a super filling $\sigma$, define $p(\sigma)$ and $m(\sigma)$ to be the numbers of positive and negative entries in $\sigma$, respectively.
\end{defn}

The notation $I(a,b)$ allows us to extend the definitions of the major index, $\inv$, and $\quinv$ to super fillings. Let $\sigma:\dg(\lambda)\rightarrow \Ac$, and define  
\[
\maj(\sigma)=\sum_{\substack{u\in\dg(\lambda)\\I(\sigma(u),\sigma(\South(u))=1}} \leg(u)+1.
\]
For $a,b,c\in\Ac\cup\{0,\infty\}$, we say $\Qc(a,b,c)=0$ if and only if exactly \emph{one} of the following is true:
\[
\{I(a,b)=1, I(c,b)=0, I(a,c)=0\}.
\]
Note that it is impossible for all three to be true, so if more than one is true, $\Qc(a,b,c)=1$. When $\sigma=|\sigma|$, this definition coincides with the inequalities in \eqref{eq:ineq}, and thus characterizes $\inv(\sigma)$ and $\quinv(\sigma)$ when $\sigma$ is a super-filling.

\begin{defn}
Let $X$ and $Y$ be two alphabets. The \emph{superization} of a symmetric function $f(X)$ is \[\widetilde{f}(X,Y) := f[X-\varepsilon Y].\]
\end{defn}
In particular, we have  
\begin{equation}\label{eq:omega}
f[X(t-1);q,t]= f[tX-\varepsilon X;q,t]=\widetilde{f}(tX,\varepsilon X;q,t).
\end{equation}

Due to \cite[Proposition 4.3]{HHL05}, if 
\[
f(X;q,t)=\sum_{\sigma:\dg(\lambda)\rightarrow \Z^+}q^{\maj(\sigma)}t^{\iota(\sigma)}x^{\sigma},\] 
where $\iota(\sigma)$ can be either $\inv(\sigma)$ or $\quinv(\sigma)$, one obtains a formula for its superization $\widetilde{f}(X,Y)=f[X-\varepsilon Y]$ as a generating function over super fillings $\sigma:\dg(\lambda)\rightarrow \Ac$:
\begin{equation}\label{eq:super}
\widetilde{f}(X,Y;q,t)=\sum_{\sigma:\dg(\lambda)\rightarrow \Ac}q^{\maj(\sigma)}t^{\iota(\sigma)}z^{\sigma}
\end{equation}
where $z_i=x_i$ if $i\in\Z^+$ and $z_i=y_i$ if $i\in\Z^-$.

Rearranging \eqref{eq:HJ} for $J_{\lambda}(X;q,t)$, we get 
\begin{align}
J_{\lambda}(X;q,t)&=t^{n(\lambda)}\HH_{\lambda}[X(1-t);q,t^{-1}]=t^{n(\lambda)+n}\HH_{\lambda}[X(t^{-1}-1);q,t^{-1}]. \label{eq:J}
\end{align}

As detailed in \cite[Lemma 5.2]{HHL05} and \cite[Section 3]{AMM20}, \eqref{eq:super} and \eqref{eq:omega} imply, respectively, 
\begin{align}
\HH_{\lambda}[X(t-1);q,t] &= \sum_{\sigma:\dg(\lambda)\rightarrow \Ac}(-1)^{m(\sigma)}q^{\maj(\sigma)}t^{p(\sigma)+\inv(\sigma)}x^{|\sigma|},\label{eq:C}\\
&= \sum_{\sigma:\dg(\lambda)\rightarrow\Ac}(-1)^{m(\sigma)}q^{\maj(\sigma)}t^{p(\sigma)+\quinv(\sigma)}x^{|\sigma|}.\label{eq:C quinv}
\end{align}
These compress to formulas for $J_{\lambda}$ in terms of $\inv$- and $\quinv$-non-attacking super-fillings, respectively:
\begin{align}
J_{\lambda}(X;q,t) &= t^{n(\lambda)+n} \sum_{\substack{\sigma:\dg(\lambda)\rightarrow \Ac\\|\sigma| \inona}}(-1)^{m(\sigma)}q^{\maj(\sigma)}t^{-p(\sigma)-\inv(\sigma)}x^{|\sigma|}, \label{eq:H inv nonattacking}\\
&= t^{n(\lambda)+n} \sum_{\substack{\sigma:\dg(\lambda)\rightarrow \Ac\\|\sigma| \qnona}}(-1)^{m(\sigma)}q^{\maj(\sigma)}t^{-p(\sigma)-\quinv(\sigma)}x^{|\sigma|}. \label{eq:H quinv nonattacking}
\end{align}

In \cite[Section 8]{HHL05}, it is shown how to further compress \eqref{eq:H inv nonattacking} to obtain \eqref{eq:P inv} by grouping together and evaluating the total contribution of all super fillings that have the same absolute value to obtain the HHL tableaux formulas in terms of $\inv$-non-attacking tableaux. In \cite{Man23}, an analogous strategy is used to compress \eqref{eq:H quinv nonattacking} to obtain the following formulas in terms of $\quinv$-non-attacking tableaux.

\begin{theorem}[{\cite[Theorem 5.3]{Man23}}]
The symmetric Macdonald polynomial $P_{\lambda}(X;q,t)$ is given by
\begin{equation}\label{eq:P quinv}
P_{\lambda}(X;q,t) =  \frac{1}{\perm_{\lambda}(t)}\hspace{-0.2in}\sum_{\substack{\sigma:\dg(\lambda)\rightarrow\Z^+\\\sigma\,\qnona}}\hspace{-0.2in}q^{\maj(\sigma)}t^{\coquinv(\sigma)}
x^{\sigma}\hspace{-0.2in} \prod_{\substack{u,~\South(u)\in\dg(\lambda)\\\sigma(u)\neq\sigma(\South(u))}}\frac{1-t}{1-q^{\leg(u)+1}t^{\rarm(u)+1}}.
\end{equation}
\end{theorem}

\begin{remark}
When $\lambda$ is a strict partition, $\perm_{\lambda}(t)=1$, and \eqref{eq:P quinv} coincides with Lenart's formula for $P_{\lambda}$ in \cite{Lenart}. It would be interesting to see if Lenart's compression of quantum alcove walks could be extended for general partitions to map to $\coquinv$-sorted non-attacking tableaux. 
\end{remark}

\begin{remark}
In comparing \eqref{eq:P quinv} to \eqref{eq:P inv}, one notices the absence of the pre-factor $\widetilde{\PR}_{\lambda}(q,t)/\PR_{\lambda}(q,t)$. This comes from the fact that 
\begin{equation}\label{eq:PR perm}
\PR_{\lambda}(q,t)=\prod_{u\in\dg(\lambda)}(1-q^{\leg(u)}t^{\arm(u)+1})=\perm_{\lambda}(t)\prod_{u\in\overline{\dg}(\lambda)}(1-q^{\leg(u)+1}t^{\rarm(u)+1}),
\end{equation}
and so all terms of the form $(1-q^{\ell}t^a)$ in the numerator are cancelled in the product $P_{\lambda}(X;q,t)=\PR_{\lambda}(q,t)^{-1}J_{\lambda}(X;q,t)$ when the $\quinv$ and $\widehat{\texttt{arm}}$ statistics are used.
\end{remark}

As a corollary we obtain a (co)quinv formula for the integral form $J_{\lambda}(X;q,t)$.
\begin{equation}\label{eq:J quinv}
J_{\lambda}(X;q,t)=\hspace{-0.4in}\sum_{\substack{\sigma:\dg(\lambda)\rightarrow\Z^+\\\sigma\,\qnona}}\hspace{-0.3in}q^{\maj(\sigma)}t^{\coquinv(\sigma)}
x^{\sigma}\hspace{-0.2in} \prod_{\substack{u\in\overline{\dg}(\lambda)\\\sigma(u)=\sigma(\South(u))}}\hspace{-0.3in}(1-q^{\leg(u)+1}t^{\rarm(u)+1})\hspace{-0.2in}\prod_{\substack{u\in\dg(\lambda)\\\sigma(u)\neq\sigma(\South(u))}}\hspace{-0.3in}(1-t).
\end{equation}
It should be noted that the first product is over cells $u\in\overline{\dg}(\lambda)$ above the bottom row, while the second product is over all cells $u\in\dg(\lambda)$, accounting for all necessary factors of $(1-t)$.

Noting the pre-factor $\perm_{\lambda}(t)^{-1}$ in the formula \eqref{eq:P quinv}, in \cite{Man23} we conjectured the compact formula \eqref{eq:main}, which uses $\coquinv$-sorted $\quinv$-non-attacking fillings, and thereby eliminates this pre-factor. We prove this conjecture in \cref{sec:operators}. 

\begin{example}\label{ex:P}
We compute $P_{2,2}(X;q,t)$ using \cref{thm:main} to get
\[
P_{(2,2)}(X;q,t) = m_{22}+\frac{(1+q)(1-t)}{1-qt}m_{211}+ \frac{(2+t+3q+q^2+3qt+2q^2t)(1-t)^2}{(1-qt)(1-qt^2)}m_{1111}.
\]
\begin{center}
\renewcommand{\arraystretch}{2}
\begin{tabular}{c c  c c c c c c c c c }
\tableau{1&2\\1&2}&\hspace{0.5in} & \tableau{1&2\\1&3} & \tableau{1&3\\1&2} &\hspace{0.5in} &  \tableau{1&2\\3&4}& \tableau{1&2\\4&3}  & \tableau{1&3\\2&4}  & \tableau{1&3\\4&2}  & \tableau{1&4\\2&3}  & \tableau{1&4\\3&2}\\
$1$ && $\frac{(1-t)}{1-qt}$ & $\frac{q(1-t)}{1-qt}$ && $1$ & $t$ & $1$ & $qt$ & $q$ & $qt$\\
 &&&&&\tableau{2&3\\1&4}  & \tableau{2&3\\4&1}  & \tableau{2&4\\1&3}  & \tableau{2&4\\3&1}  & \tableau{3&4\\1&2}  & \tableau{3&4\\2&1}\\
 && & && $qt$ & $q$ & $q^2t$ & $q$ & $q^2$ & $q^2t$
\end{tabular}
\end{center}
The common factor of $\frac{(1-t)^2}{(1-qt)(1-qt^2)}$ was left out from the weights of the fillings for $m_{1111}$.
\end{example}

\subsection{Compact formulas for $\HH_{\lambda}$ via \texttt{inv}-flip and \texttt{quinv}-flip operators $\tau_j$ and $\rho_j$}\label{sec:compact}

\begin{defn}\label{def:perm}
Let $\lambda=\langle 1^{m_1}2^{m_2}\cdots k^{m_k}\rangle$ be a partition. We define $\perm_{\lambda}(t)$ to be the $t$-analog of the stabilizer of $\lambda$ in $S_{\ell(\lambda)}$:
\[
\perm_{\lambda}(t):=[m_1]_t!\cdots [m_k]_t!
\]
where $[k]_t:=1+\cdots+t^{k-1}$ and $[k]_t!:=[k]_t[k-1]_t\cdots[1]_t$.
\end{defn}

We generalize the above to define $\perm(\sigma)$ where $\sigma$ is a filling of a partition shape. 
\begin{defn}\label{def:perm sigma}
Let $\sigma^{(i)}$ be a filling of a $m\times i$ rectangular diagram ($m$ columns of length $i$) such that there are $k$ distinct columns with multiplicities $\{m_1,\ldots,m_k\}$, so that $m_1+\cdots+m_k=m$. We define the statistic $\perm(\sigma^{(i)})$ to be the inversion generating function of a word in the letters $\langle 1^{m_1}\cdots k^{m_k}\rangle$:
\[
\perm(\sigma^{(i)}) = {m \brack m_1,\ldots,m_k}_t:=\frac{[m]_t!}{[m_1]_t!\cdots[m_k]_t!}.
\]
Then, for a filling $\sigma:\dg(\lambda)\rightarrow \Z^+$, for $1\leq i\leq \lambda_1$, let $\sigma^{(i)}$ be the (possibly empty) rectangular block consisting of the columns of height $i$ of $\sigma$ (per convention, $\perm(\emptyset)=1$). Define
\[
\perm(\sigma) = \prod_{i=1}^{\lambda_1} \perm(\sigma^{(i)}).
\]
Described in another way, $\perm(\sigma)$ is the $t$-analog of the number of distinct ways of permuting the columns of $\sigma$ within the shape $\dg(\lambda)$. 
\end{defn}

In particular, if $\sigma$ is a ($\inv$ or $\quinv$) non-attacking filling of $\dg(\lambda)$, all of its columns are necessarily distinct, and so $\perm(\sigma)=\perm_{\lambda}(t)$.

\begin{example}
We compute $\perm(\sigma)$ for the tableau $\sigma$ from \cref{ex:stats}. There are two distinct columns of height 1, each appearing twice, so the multiplicities are $\{2,2\}$ and thus 
\[\perm(\sigma^{(1)})={4\brack 2,2}_t=[4]_t[3]_t/[2]_t.\] For height 2, there are three columns with multiplicities $\{2,1\}$, giving \[\perm(\sigma^{(2)})={3\brack 2,1}_t=[3]_t.\] 
Finally, for height 3, there are two distinct columns with multiplicities $\{1,1\}$, so 
\[\perm(\sigma^{(3)})={2\brack 1,1}_t=[2]_t.\] 
Thus $\perm(\sigma) =\prod_i\perm(\sigma^{(i)})=[4]_t[3]_t^2$.
\end{example}

By choosing a canonical way of sorting columns of equal height in a filling, the formulas \eqref{eq:H inv} and \eqref{eq:H quinv} for $\HH_{\lambda}$ can be written more compactly. This was first done in \cite[Theorem 3.5]{CHMMW20} for the $\inv$ formula, and was later replicated in \cite[Theorem 4.2]{Man23} for the $\quinv$ analog. 
\begin{theorem}\label{thm:compact}
The modified Macdonald polynomial can be obtained as
\begin{align}
\label{eq:H compact inv}
\HH_{\lambda}(X;q,t)& = \sum_{\substack{\sigma:\dg(\lambda)\rightarrow\Z^+\\\sigma\,\isort}} \perm(\sigma)x^{\sigma}q^{\maj(\sigma)}t^{\inv(\sigma)},\\
\label{eq:H compact quinv}
&= \sum_{\substack{\sigma:\dg(\lambda)\rightarrow\Z^+\\\sigma\,\qsort}} \perm(\sigma)x^{\sigma}q^{\maj(\sigma)}t^{\quinv(\sigma)}.
\end{align}
where the first sum is over $\inv$-\emph{sorted} fillings and the second sum is over $\quinv$-\emph{sorted} fillings according to \cref{def:sorted}.
\end{theorem}
We outline the strategy for proving \cref{thm:compact} as a warm-up for the proof of \cref{thm:main}.

\subsubsection{The operators $\tau_i$ and $\rho_i$}
The operator $\tau_i$ was introduced in \cite{LoehrNiese} to act on fillings as an \emph{inversion flip operator}, characterized by the following properties: it preserves the content and the major index, and when it acts nontrivially, it changes the number of inversions by exactly 1. 

For a partition $\lambda$, we say $1\leq i\leq \ell(\lambda)$ is \emph{$\lambda$-compatible} if $\lambda_i=\lambda_{i+1}$. We say that a transposition $s_i$ is \emph{$\lambda$-compatible} if the index $i$ is $\lambda$-compatible. Moreover, we call two indices $i<j$ \emph{$\lambda$-comparable} if $\lambda_i=\lambda_j$. 

For a word $w = w_1 w_2 \cdots w_k \in \Z_{>0}^k$, the \emph{left action} of the simple transposition $s_i \in S_k$ (for $1 \le i < k$) is defined by swapping the entries at positions $i$ and $i+1$:
\[
s_i \cdot w = w_1 \cdots w_{i-1} w_{i+1} w_i w_{i+2} \cdots w_k.
\]

Then, we define the set of $\lambda$-compatible permutations of $w$, denoted $\Sym_{\lambda}(w)$, to be the set of words obtained by applying a sequence of $\lambda$-compatible transpositions to $w$ (applied from right to left):
\[
\Sym_{\lambda}(w)=\{w'\in\Z_{>0}^k:w'=s_{i_u}\cdots s_{i_1}\cdot w,\ \text{with each $s_{i_j}$\ $\lambda$-compatible for $1\leq j\leq u$}\}.
\]

For a filling $\sigma$ of $\dg(\lambda)$ and $\lambda$-compatible $i$, define the operator $\T_i^{(r)}$ to act on $\sigma$ by swapping the entries of the squares $(r,i)$ and $(r,i+1)$. For $1\leq r,s\leq \lambda_i$, the compact notation $\T_i^{[r,s]}:=\T_i^{(r)}\circ \T_i^{(r+1)}\circ\cdots\circ\T_i^{(s)}$ represents a sequence of swaps of entries in consecutive rows between the same pair of adjacent columns. Note that since the $\T_i^{(r)}$'s act on rows independently, they commute, and so $\T_i^{[r,s]}=\T_i^{[s,r]}$. See \cref{ex:trs}.

\begin{example}\label{ex:trs}
For example, when $\lambda=(3,3,3,2,2)$, the $\lambda$-compatible indices are $\{1,2,4\}$. For $\sigma$ below, we compute $\T_2^{[2,3]}(\sigma)$, $\T_4^{(2)}(\sigma)$, and $\T_4^{[1,2]}(\sigma)$, with the swapped entries shown in bold:
\[
\sigma=\raisebox{10pt}{\tableau{2&1&3\\2&1&2&1&1\\1&2&1&2&3}}\,,\qquad \T_2^{[2,3]}(\sigma)=\raisebox{10pt}{\tableau{2&\mathbf{3}&\mathbf{1}\\2&\mathbf{2}&\mathbf{1}&1&1\\1&2&1&2&3}}\,,\qquad \T_4^{(2)}(\sigma)=\sigma, \quad \text{and}\  \T_4^{[1,2]}(\sigma)=\raisebox{10pt}{\tableau{2&2&3\\2&1&2&\mathbf{1}&\mathbf{1}\\1&2&1&\mathbf{3}&\mathbf{2}}}\,.
\]
\end{example}

\begin{defn}[The operator $\tau_i$]\label{def:tau}
For a partition $\lambda$ and a $\lambda$-compatible index $i$, let $\sigma:\dg(\lambda)\rightarrow \Z^+$. If columns $i$ and $i+1$ are identical in $\sigma$, $\tau_i(\sigma)=\sigma$. Otherwise, let $\ell'$ be minimal such that $\sigma(\ell',i)\neq \sigma(\ell',i+1)$, and let $\ell'\geq h'\leq \lambda_i$ be minimal such that $\Qc(\sigma(h'+1,i),\sigma(h',i),\sigma(h'+1,i+1))=\Qc(\sigma(h'+1,i+1),\sigma(h',i),\sigma(h'+1,i+1))$, where per convention, $\sigma(\lambda_i+1,i)=0$ for all $i$. Then $\tau_i(\sigma)=\T_i^{[\ell',h']}(\sigma)$, swapping the entries in rows $\ell'$ through $h'$ between columns $i$ and $i+1$, while keeping all other entries unchanged. See \cref{ex:tau}.
\end{defn}

In \cite{AMM20}, we defined analogous operators for the $\quinv$ statistic. These operators were also originally called $\tau_i$, but we shall assign to them the new name $\op_i$ to distinguish from the original $\tau_i$'s.

\begin{defn}[The operators $\op_i$]\label{def:rho}
For a partition $\lambda$ and a $\lambda$-compatible index $i$, let $\sigma:\dg(\lambda)\rightarrow \Z^+$. If columns $i$ and $i+1$ are identical in $\sigma$, $\op_i(\sigma)=\sigma$. Otherwise, let $\ell$ be maximal such that $\sigma(\ell,i)\neq \sigma(\ell,i+1)$, and let $1\geq h\leq \ell$ be maximal such that $\Qc(\sigma(h,i),\sigma(h-1,i),\sigma(h-1,i+1))=\Qc(\sigma(h,i+1),\sigma(h-1,i),\sigma(h-1,i+1))$ where per convention, $\sigma(0,i)=\infty$ for all $i$. Then $\op_i(\sigma)=\T_i^{[h,\ell]}(\sigma)$, swapping the entries in rows $h$ through $\ell$ between columns $i$ and $i+1$, while keeping all other entries unchanged. See \cref{ex:tau}.
\end{defn}

Another description of $\rho_i$, which will be useful for comparison to the definitions in \cref{sec:operators}, is as follows. Define $\op_i(\sigma):=\op_i^{(\lambda_i)}(\sigma)$, where the operator $\op_i^{(r)}$ is defined recursively to act on $\sigma$ as follows based on the contents $a=\sigma(r,i)$, $b=\sigma(r,i+1)$, $c=\sigma(r-1,i)$, $d=\sigma(r-1,i+1)$ in the configuration
\vspace{-0.2in}
\begin{center} \begin{tikzpicture}[scale=0.5]
    \node at (0,0) {\tableau{a&b\\
               c&d }};
               \node at (-.5,1.5) {\tiny $i$};
               \node at (.5,1.5) {\tiny $i+1$};
               \node at (-1.7,-.5) {\tiniest{$r-1$}};
               \node at (-1.5,.5) {\tiniest{$r$}};
               \end{tikzpicture}
               \end{center}
\begin{itemize}
\item[i.] If $\sigma(r,i)=\sigma(r,i+1)$, $\op_i^{(r)}(\sigma)=\op_i^{(r-1)}(\sigma)$.
\item[ii.] Otherwise, if $\Qc(a,c,d)=\Qc(b,c,d)$, $\op_i^{(r)}(\sigma)=\T_i^{(r)}(\sigma)$, and
\item[iii.] If $\Qc(a,c,d)=\Qc(b,c,d)$, $\op_i^{(r)}(\sigma)=\op_i^{(r-1)}(\T_i^{(r)}(\sigma))$.
\end{itemize}
The intuition behind this definition is that $\op_i$ scans the rows of $\sigma$ in columns $i,i+1$ from top to bottom, swaps the first pair of differing entries between the two columns (which changes the contribution to $\quinv$ from the affected $L$-triple by $\pm1$), and then sequentially swaps the entries in the rows below to prevent any additional changes to $\quinv$ coming from the $L$-triples in the rows below. There is a similar recursive definition for $\tau_i$, which we will omit.

\begin{defn}\label{def:sorted}
For a partition $\lambda$, we say $\sigma:\dg(\lambda)\rightarrow\Z^+$ is $\inv$-\emph{sorted} if $\inv(\sigma)\leq \inv(\tau_i(\sigma))$ for any $\lambda$-compatible $i$. If $\sigma$ is $\inv$-non-attacking, we say it is $\coinv$-\emph{sorted} if and only if every pair of adjacent entries in the bottom-row cells $(1,i),(1,i+1)$ for $\lambda$-compatible $i$ is increasing from left to right (meaning that $\coinv(\sigma)\leq \coinv(\tau_i(\sigma))$). Analogously, we say $\sigma$ is $\quinv$-\emph{sorted} if $\quinv(\sigma)\leq \quinv(\op_i(\sigma))$ for any $\lambda$-compatible $i$.  If $\sigma$ is $\quinv$-non-attacking, we say it is $\coquinv$-\emph{sorted} if and only if every pair of topmost adjacent entries $(\lambda_i,i),(\lambda_{i},i+1)$ for $\lambda$-compatible $i$ is increasing from left to right (meaning that $\coquinv(\sigma)\leq \coquinv(\tau_i(\sigma))$). See \cref{ex:sorted}.
\end{defn}

\begin{example}\label{ex:tau}
Suppose $\sigma$ has columns $i,i+1$ as shown below, with $\lambda_i=6$. For $\tau_i(\sigma)$, we have $\ell'=1$, and $h'=2$, since applying $\T_i^{(2)}$ changes the triple at rows $h',h'+1$ from $\tableau{2&3\\3}$ to $\tableau{2&3\\4}$, which is not an $\inv$ triple in both cases.  For $\op_i(\sigma)$, we have $\ell=5$, and $h=3$, since applying $\T_i^{(3)}$ changes the triple at rows $h,h-1$ from $\tableau{2\\3&4}$ to $\tableau{3\\3&4}$, which is a $\quinv$ triple in both cases. 
\[
\begin{tikzpicture}[scale=.5]
\node at (-2,-2) {$\sigma=$};
\cell00{2}\cell01{2}
\cell10{3}\cell11{4}
\cell20{2}\cell21{3}
\cell30{2}\cell31{3}
\cell40{3}\cell41{4}
\cell50{1}\cell51{3}
\node at (-0.5,-5.5) {\tiny $i$};
\node at (0.5,-5.5) {\tiny $i+1$};

\begin{scope}[shift={(3,0)}]
\node at (3.5,-2) {$\tau_j(\sigma)=$};
\node at (7.5,-4.5) {\small$\ell'$};
\node at (7.5,-3.5) {\small $h'$};
\cell06{2}\cell07{2}
\cell16{3}\cell17{4}
\cell26{2}\cell27{3}
\cell36{2}\cell37{3}
\cell46{4}\cell47{3}
\cell56{3}\cell57{1}
\node at (5.5,-5.5) {\tiny $i$};
\node at (6.5,-5.5) {\tiny $i+1$};
\end{scope}

\begin{scope}[shift={(12,0)}]
\node at (3.5,-2) {$\rho_j(\sigma)=$};
\node at (7.5,-.5) {\small$\ell$};
\node at (7.5,-2.5) {\small $h$};
\cell06{2}\cell07{2}
\cell16{4}\cell17{3}
\cell26{3}\cell27{2}
\cell36{3}\cell37{2}
\cell46{3}\cell47{4}
\cell56{1}\cell57{3}
\node at (5.5,-5.5) {\tiny $i$};
\node at (6.5,-5.5) {\tiny $i+1$};
\end{scope}
\end{tikzpicture}
\]

\end{example}

\begin{example}\label{ex:sorted}
We show an $\inv$-sorted tableau $\sigma_1$, a $\coinv$-sorted $\inv$-non-attacking tableau $\sigma_2$, a $\quinv$-sorted tableau $\sigma_3$, and a $\coquinv$-sorted $\quinv$-non-attacking tableau $\sigma_4$:
\[
\sigma_1=\raisebox{10pt}{\tableau{2&1&3\\2&2&1&1&3\\1&1&2&3&3}},\qquad \sigma_2=\raisebox{10pt}{\tableau{3&2&4\\1&2&5&4&6\\7&3&1&5&4}},\qquad \sigma_3=\raisebox{10pt}{\tableau{3&2&2\\2&2&1&1&1\\1&1&2&3&2}},\qquad \sigma_4=\raisebox{10pt}{\tableau{2&4&6\\7&3&1&5&6\\3&2&1&4&6}}
\]
\end{example}

The following lemma establishes the key properties of the $\tau_i$'s and the $\op_i$'s, which are to preserve the $\coinv$ and change the $\inv$ and $\quinv$, respectively, in a controlled way.
\begin{lemma}[{\cite[Lemmas 3.9, 3.10, 3.11]{CHMMW20} and \cite[Lemmas 7.5 and 7.6]{AMM20}}]\label{thm:tau}
Let $\lambda$ be a partition, let $i$ be $\lambda$-compatible, and let $\sigma:\dg(\lambda)\rightarrow \Z^+$.
\begin{itemize}
\item[i.] $\tau_i$ and $\op_i$ are involutions.
    \item[ii.] If columns $i,i+1$ of $\sigma$ are not identical, then $\inv(\tau_i(\sigma))=\quinv(\sigma)\pm 1$ and $\quinv(\op_i(\sigma))=\quinv(\sigma)\pm 1$.
    \item[iii.] $\maj(\tau_i(\sigma))=\maj(\op_i(\sigma))= \maj(\sigma)$.
    \end{itemize}
\end{lemma}
\eqref{eq:H compact inv} is proved in \cite{CHMMW20} by using the $\tau_i$'s to generate the entire set of fillings $\sigma:\dg(\lambda)\rightarrow\Z^+$ from the set of  $\inv$-sorted fillings, so that each  $\inv$-sorted tableau $\nu$ generates a set of fillings with total weight equal to $x^{\nu}t^{\inv(\nu)}q^{\maj(\nu)}\perm(\nu)$. Similarly, \eqref{eq:H compact quinv} is proved in \cite{Man23} following an identical strategy, by generating the entire set of fillings $\sigma:\dg(\lambda)\rightarrow\Z^+$ from applying $\op_i$'s to the set of $\quinv$-sorted fillings, so that each $\quinv$-sorted tableau $\nu$ generates a set of fillings with total weight equal to $x^{\nu}t^{\quinv(\nu)}q^{\maj(\nu)}\perm(\nu)$. We need one final technical definition to execute these strategies.

\subsubsection{Positive distinguished subexpressions} 
Using the fact that the symmetric group $S_{\ell}$ is generated by simple transpositions $\{s_1,\ldots,s_{\ell-1}\}$, we can use compositions of $\tau_i$'s or $\op_i$'s to generate the set of all possible $\lambda$-compatible permutations of the columns of a given filling $\sigma$ of $\dg(\lambda)$ where $\ell=\ell(\lambda)$.  Unfortunately, in general, neither the $\tau_i$'s or the $\op_i$'s satisfy braid relations. Thus one must choose a canonical way to compose sequences of $\tau_i$'s (resp.~$\op_i$'s) to that each filling in $\dg(\lambda)\rightarrow \Z^+$ is generated by a unique  $\inv$-sorted (resp.~$\quinv$-sorted) filling via a well-defined sequence of operators. This is done for the $\tau_j$'s using \emph{positive distinguished subexpressions} (PDS) in \cite[Section 3.2]{CHMMW20}. 

The notion of a PDS of a reduced expression originates from \cite{MarshRietsch}. For a fixed $n$, the corresponding set of PDS is a subset of the reduced words generated by the simple transpositions $\{s_1,\ldots,s_{n-1}\}$, which forms a spanning tree on the Bruhat lattice. In other words, for any permutation $\pi\in S_n$, the PDS $s_{i_k}\cdots s_{i_1}$ corresponding to $\pi$ is the reduced word coming from the unique path from the identity permutation $id$ to $\pi$ on this spanning tree, so that $\pi=s_{i_k}\cdots s_{i_1}\circ id$. There is not a unique choice for the set of PDS, but for the sake of concreteness, we will use the same set of PDS as is described in \cite[Section 3.2]{CHMMW20}, which we present in the following definition. 

\begin{defn}\label{def:PDS}
Fix $n$ and let $w_0=(n,n-1,\ldots,2,1)$. Fix the reduced expression for $w_0$ to be $w_0=s_1(s_2s_1)\cdots(s_{n-1}\cdots s_2s_1)=s_{i_t},\ldots,s_{i_1}$ (with $t={n\choose 2}$). For a permutation $\pi\in S_n$, define $\PDS(\pi)=v_t\cdots v_1$ where $v_j\in\{s_{i_j},e\}$, defined as follows. For two permutations $\alpha,\beta\in S_n$, we write $\alpha<\beta$ if $\ell(\alpha)<\ell(\beta)$. Set $v_{(0)}=\pi$, and set
\[
v_{(j)}=\begin{cases} v_{(j-1)}s_{i_j}&\text{if}\ v_{(j-1)}s_{i_j}<v_{(j-1)}\\
v_{(j-1)}&\text{otherwise}.\end{cases}
\]
Then $v_j=s_{i_j}$ if $v_{(j-1)}s_{i_j}<v_{(j-1)}$ and $v_j=e$ otherwise. Finally, for a permutation $\lambda$, we define $\PDS(\lambda):=\{\PDS(\alpha):\alpha\cdot \lambda=\lambda\}$ to be the set of PDS corresponding to the $\lambda$-compatible permutations. For a filling $\sigma$ of $\dg(\lambda)$, we define $\PDS(\sigma)$ to be the subset of $\PDS(\lambda)$ that corresponds to all possible permutations of the columns of $\sigma$ within the shape $\dg(\lambda)$.
\end{defn}

Now let $\nu:\dg(\lambda)\rightarrow\Z^+$ be a  $\inv$-sorted (resp.~$\quinv$-sorted filling). We shall define a family of tableaux $\{\tau_i\}\nu$ (resp.~$\{\op_i\}\nu$), generated by applying sequences of $\tau_i$'s (resp.~$\op_i$'s) corresponding to $\PDS(\nu)$, the set of $\nu$-compatible PDS. Then,
\begin{equation}\label{eq:decomp}
\{\sigma:\dg(\lambda)\rightarrow\Z^+\} = \biguplus_{\substack{\sigma:\dg(\lambda)\rightarrow\Z^+\\\sigma\,\isort}} \{\tau_i\}\sigma=\biguplus_{\substack{\sigma:\dg(\lambda)\rightarrow\Z^+\\\sigma\,\qsort}} \{\op_i\}\sigma.
\end{equation}
If $\nu$ is a $\quinv$-sorted filling, let $\sigma=\op_{i_k}\circ\cdots\circ\op_{i_1}(\sigma)\in\{\op_j\}\nu$, where $s_{i_k}\cdots s_{i_1}\in\PDS(\nu)$. By \cref{thm:tau}, $\quinv(\sigma)=t^k\quinv(\nu)$. (Morally, this sequence of operators is permuting the columns of $\nu$ by the permutation $s_{i_k}\cdots s_{i_1}$.) The length (i.e.~inversion) generating function of the set $\PDS(\sigma)$ is precisely equal to $\perm(\sigma)$. Thus we obtain
\[\sum_{\sigma\in\{\op_j\}\nu}x^{\sigma}q^{\maj(\sigma)}t^{\quinv(\sigma)} = x^{\nu}q^{\maj(\nu)}t^{\quinv(\nu)}\perm(\nu).\]
A similar equality holds for $\nu$ that is  $\inv$-sorted with $\tau_i$ replacing $\op_i$ and  $\inv$ replacing $\quinv$ above. Combined with \eqref{eq:decomp}, this proves \cref{thm:compact}.

\section{The probabilistic operators $\widetilde{\op}_i$}\label{sec:operators}
Unfortunately, the operators $\op_i$ cannot be used directly to prove \cref{thm:main}. They are not well-defined on $\quinv$-non-attacking tableaux as they generally fail to preserve the non-attacking condition; moreover, they may not maintain the weight in the $\widehat{\texttt{arm}}$  statistic. To deal with this, we define a \emph{probabilistic operator} $\widetilde{\op}_i$ on $\quinv$-non-attacking tableaux that will take the place of the $\op_i$'s.

Denote by $\Tab(\lambda)$ the set of tableaux $\sigma:\dg(\lambda)\rightarrow\Z^+$ that are $\quinv$-non-attacking.

\begin{defn}
Let $\sigma\in\Tab(\lambda)$ and set $k=\ell(\lambda)$. The \emph{border} of $\sigma$, denoted $\Top(\sigma)$, is a word in $\Z_{>0}^k$ given by the sequence of entries at the tops of the columns of $\sigma$, read from left to right:
\[
\Top(\sigma)=\sigma(\lambda_1,1)\sigma(\lambda_2,2)\cdots \sigma(\lambda_k,k).
\]
Define $\inc_{\lambda}(w)\in\Sym_{\lambda}(w)$ to be the unique rearrangement of the letters of $\Top(\sigma)$ such that the entries are strictly increasing whenever possible: $\inc_{\lambda}(w)_i<\inc_{\lambda}(w)_{i+1}$ for all $\lambda$-comparable $i$. Then for any $w\in\Z_{>0}^k$ that can appear as the border of some non-attacking filling of $\dg(\lambda)$, we define the \emph{length of $w$ with respect to $\lambda$} to be the number of $\lambda$-comparable inversions: 
\[\ell_{\lambda}(w):=\{i<j: \lambda_i=\lambda_j,\ w_i>w_j\}.
\]  
Equivalently, $\ell_{\lambda}(w)=\ell(\pi)$, where $\pi\in S_k$ such that $\pi(\inc_{\lambda}(w))=w$, and $\ell(\pi)$ is the length of the permutation $\pi$. In particular, $\inc_\lambda(w)$ is the unique element of $\Sym_\lambda(w)$ such that $\ell_{\lambda}(\inc_{\lambda}(w))=0$. Moreover, for any $w$ that can appear as the border of some filling in $\Tab(\lambda)$,
\begin{equation}\label{eq:perm}
\sum_{v\in\Sym_{\lambda}(w)}t^{\ell_{\lambda}(v)}=[m_1]_t!\cdots [m_k]_t!=\perm_{\lambda}(t),
\end{equation}
where $\lambda=\langle 1^{m_1}2^{m_2}\cdots k^{m_k}\rangle$. See \cref{ex:prob tau} for an example of $\Sym_{\lambda}$, $\inc_\lambda$, and $\ell_{\lambda}$.
\end{defn}

Observe that $\sigma\in\Tab(\lambda)$ with $\Top(\sigma)=w$ is $\coquinv$-sorted if and only if $w=\inc_{\lambda}(w)$.

\begin{defn}
Let $\sigma\in\Tab(\lambda)$, and let $1\leq i\leq \ell(\lambda)$ be $\lambda$-compatible with $L=\lambda_i=\lambda_{i+1}$. Define the operator $\widetilde{\op}_i(\sigma)=\widetilde{\op}_i^{(L)}(\sigma)$ where for $1\leq r\leq \lambda_i$, $\widetilde{\op}_i^{(r)}$ acts on $\sigma$ by probabilistically mapping to a set of tableaux $\widetilde{\op}_i^{(r)}(\sigma)\subseteq \Tab(\lambda)$, based on the following cases for the entries $a=\sigma(r,i)$, $b=\sigma(r,i+1)$, $c=\sigma(r-1,i)$, and $d=\sigma(r-1,i+1)$ (assume all configurations are $\quinv$-non-attacking):
\begin{center}
(i) \raisebox{-45pt}{\begin{tikzpicture}[scale=0.4]\cell{-1}0{}\cell{-1}1{}
\node at (-.5,1.4) {$a$};\node at (.5,1.5) {$b$};
\node at (-.5,.5) {$\emptyset$};\node at (.5,.5) {$\emptyset$};
\node at (-1.5,1.5) {\tiniest{$1$}};\node at (-1.5,.5) {\tiniest{$0$}};\node at (-.4,2.5) {\tiniest{$i$}};\node at (.6,2.5) {\tiniest{$i+1$}};
\draw[->] (0,-.1)--(0,-.9);
\cell{2}0{}\cell{2}1{}
\node at (-.5,-1.5) {$b$};\node at (.5,-1.6) {$a$};
\node at (-.5,-2.5) {$\emptyset$};\node at (.5,-2.5) {$\emptyset$};
\node at (-1.5,-1.5) {\tiniest{$1$}};\node at (-1.5,-2.5) {\tiniest{$0$}};
\draw[blue] (-.9,-1.9) rectangle (.9,-1.1);
\end{tikzpicture}}
\quad(ii) \raisebox{-45pt}{\begin{tikzpicture}[scale=0.4]\cell00{}\cell01{}\cell{-1}0{}\cell{-1}1{}
\node at (-.5,1.4) {$a$};\node at (.5,1.5) {$b$};
\node at (-.5,.5) {$b$};\node at (.5,.5) {$d$};
\node at (-1.5,1.5) {\tiniest{$r$}};\node at (-2,.5) {\tiniest{$r-1$}};\node at (-.5,2.5) {\tiniest{$i$}};\node at (.5,2.5) {\tiniest{$i+1$}};
\draw[->] (0,-.1)--(0,-.9);
\cell{2}0{}\cell{2}1{}
\cell{3}0{}\cell{3}1{}
\node at (-.5,-1.5) {$b$};\node at (.5,-1.6) {$a$};
\node at (-.5,-2.5) {$b$};\node at (.5,-2.5) {$d$};
\node at (-1.5,-1.5) {\tiniest{$r$}};\node at (-2,-2.5) {\tiniest{$r-1$}};
\draw[blue] (-.9,-1.9) rectangle (.9,-1.1);
\end{tikzpicture}} or\ \ 
\raisebox{-45pt}{\begin{tikzpicture}[scale=0.4]
\cell00{}\cell01{}\cell{-1}0{}\cell{-1}1{}
\node at (-.5,1.4) {$a$};\node at (.5,1.5) {$b$};
\node at (-.5,.4) {$c$};\node at (.5,.5) {$d$};
\node at (-.5,2.5) {\tiniest{$i$}};\node at (.5,2.5) {\tiniest{$i+1$}};
\draw[->] (0,-.1)--(0,-.9);
\cell{2}0{}\cell{2}1{}
\cell{3}0{}\cell{3}1{}
\node at (-.5,-1.5) {$b$};\node at (.5,-1.6) {$a$};
\node at (-.5,-2.6) {$c$};\node at (.5,-2.5) {$d$};
\draw[blue] (-.9,-1.9) rectangle (.9,-1.1);
\end{tikzpicture}}
\quad (iii) \raisebox{-45pt}{\begin{tikzpicture}[scale=0.4]
\cell00{}\cell01{}\cell{-1}0{}\cell{-1}1{}
\node at (-.5,1.4) {$a$};\node at (.5,1.5) {$b$};
\node at (-.5,.4) {$c$};\node at (.5,.5) {$b$};
\node at (-1.5,1.5) {\tiniest{$r$}};\node at (-2,.5) {\tiniest{$r-1$}};\node at (-.5,2.5) {\tiniest{$i$}};\node at (.5,2.5) {\tiniest{$i+1$}};
\draw[->] (0,-.1)--(0,-.9);
\cell{2}0{}\cell{2}1{}
\cell{3}0{}\cell{3}1{}
\node at (-.5,-1.5) {$b$};\node at (.5,-1.6) {$a$};
\node at (-.5,-2.5) {$b$};\node at (.5,-2.6) {$c$};
\node at (-1.5,-1.5) {\tiniest{$r$}};\node at (-2,-2.5) {\tiniest{$r-1$}};
\draw[blue] (-.9,-1.9) rectangle (.9,-1.1);
\draw[blue] (-.9,-2.9) rectangle (.9,-2.1);
\end{tikzpicture}} or\ \ 
\raisebox{-45pt}{\begin{tikzpicture}[scale=0.4]
\cell00{}\cell01{}\cell{-1}0{}\cell{-1}1{}
\node at (-.5,1.4) {$a$};\node at (.5,1.5) {$b$};
\node at (-.5,.5) {$b$};\node at (.5,.5) {$d$};
\node at (-.5,2.5) {\tiniest{$i$}};\node at (.5,2.5) {\tiniest{$i+1$}};
\draw[->] (0,-.1)--(0,-.9);
\cell{2}0{}\cell{2}1{}
\cell{3}0{}\cell{3}1{}
\node at (-.5,-1.5) {$b$};\node at (.5,-1.6) {$a$};
\node at (-.5,-2.5) {$d$};\node at (.5,-2.6) {$c$};
\draw[blue] (-.9,-1.9) rectangle (.9,-1.1);
\draw[blue] (-.9,-2.9) rectangle (.9,-2.1);
\end{tikzpicture}} 
\quad (iv)\hspace{-0.2in} \raisebox{-45pt}{\begin{tikzpicture}[scale=0.4]
\cell00{}\cell01{}\cell{-1}0{}\cell{-1}1{}
\node at (-.5,1.4) {$a$};\node at (.5,1.5) {$b$};
\node at (-.5,.4) {$a$};\node at (.5,.5) {$d$};
\node at (-1.5,1.5) {\tiniest{$r$}};\node at (-2,.5) {\tiniest{$r-1$}};\node at (-.5,2.5) {\tiniest{$i$}};\node at (.5,2.5) {\tiniest{$i+1$}};
\draw[->] (-.5,-.1)--(-1.5,-.9);
\cell{2}{-1}{}\cell{2}{-2}{}
\cell{3}{-1}{}\cell{3}{-2}{}
\node at (-2.5,-1.5) {$b$};\node at (-1.5,-1.6) {$a$};
\node at (-2.5,-2.6) {$a$};\node at (-1.5,-2.5) {$d$};
\node at (-3.5,-1.5) {\tiniest{$r$}};\node at (-4,-2.5) {\tiniest{$r-1$}};
\draw[blue] (1.1-4,-1.9) rectangle (2.9-4,-1.1);
\draw[->] (.5,-.1)--(1.5,-.9);
\cell{2}{2}{}\cell{2}{3}{}
\cell{3}{2}{}\cell{3}{3}{}
\node at (1.5,-1.5) {$b$};\node at (2.5,-1.6) {$a$};
\node at (1.5,-2.5) {$d$};\node at (2.5,-2.6) {$a$};
\draw[blue] (1.1,-1.9) rectangle (2.9,-1.1);
\draw[blue] (1.1,-2.9) rectangle (2.9,-2.1);
\end{tikzpicture}} 
\end{center}

\begin{itemize}
\item[(i)] If $r=1$, $\widetilde{\op}_i^{(r)}$ produces $\T_i^{(r)}(\sigma)$ with probability 1.
\item[(ii)] If $b=c$ (left) or $\Qc(a,c,d)=\Qc(b,c,d)$ and $\{a,b\}\cap \{c,d\}=\emptyset$ (right), $\widetilde{\op}_i^{(r)}$ produces $\T_i^{(r)}(\sigma)$ with probability $1$.
\item[(iii)] If $b=d$ (left) or $\Qc(a,c,d)\neq\Qc(b,c,d)$ and $\{a,b\}\cap \{c,d\}=\emptyset$ (right), $\widetilde{\op}_i^{(r)}$ produces the set $\widetilde{\op}_i^{(r-1)}(\T_i^{(r)}(\sigma))$ with probability 1.
\item[(iv)] If $a=c$, set $A:=\rarm(r,i+1)+1$ and $\ell:=\leg(r,i)+1=L-r+1$. Then $\widetilde{\op}_i^{(r)}$ produces 
\begin{itemize}[label={$\bullet$}]
\item$\T_i^{(r)}(\sigma)$ with probability $\displaystyle(q^{\ell}t^{A})^{\Qc(b,a,d)}\frac{1-t}{1-q^{\ell}t^{A+1}}$, and 
\item the set $\widetilde{\op}_i^{(r-1)}(\T_i^{(r)}(\sigma))$ with probability $\displaystyle t^{1-\Qc(b,a,d)}\frac{1-q^{\ell}t^{A}}{1-q^{\ell}t^{A+1}}$.
\end{itemize}
\end{itemize}
Define the operator $\widetilde{\op}_i$ to act on $\sigma\in\Tab(\lambda)$ by producing the set of tableaux $\widetilde{\op}_i(\sigma):=\widetilde{\op}^{(L)}_i(\sigma)$. Suppose $\sigma'\in\widetilde{\op}_i(\sigma)$ such that $\sigma'=\T_i^{[k,L]}(\sigma)$. Let $\prob_i^{(r)}(\sigma,\sigma')$ denote the probability that $\widetilde{\op}^{(r)}_i$ applied to $\T_i^{[r+1,L]}(\sigma)$ produces $\widetilde{\op}^{(r-1)}_i\left(\T_i^{[r,L]}(\sigma)\right)$ when $r>k$ and $\T_i^{[r,L]}(\sigma)$ when $r=k$, and call this the \emph{transition probability at the $r$'th row}. Then the probability that $\widetilde{\op}_i$ applied to $\sigma$ produces $\sigma'$ is
\[
\prob_i(\sigma, \sigma')=\prod_{r=k}^{L} \prob_i^{(r)}(\sigma, \sigma').
\]
Equivalently, $\prob_i^{(r)}(\sigma, \sigma')$ can be computed by comparing the entries of rows $r,r-1$ in columns $i,i+1$ of $\sigma$ and $\sigma'$ and identifying which of the cases (i)-(iv) applies. Note that we may drop the index $i$ in $\prob_i(\sigma,\sigma')$ since it is uniquely determined. See \cref{ex:prob tau} for a sample computation.
\end{defn}

\begin{lemma}\label{lem:welldefined}
Let $\sigma\in\Tab(\lambda)$ and let $i$ be $\lambda$-compatible. Then
\[
\sum_{\sigma'\in\widetilde{\op}_i(\sigma)}\prob_i(\sigma,\sigma')=1.
\]
\end{lemma}

\begin{proof}
This follows from a straightforward verification that the sum of the transition probabilities at the $r$'th row when $\widetilde{\op}_i^{(r)}$ is applied is equal to 1. Indeed, this only needs to be checked for case (iv), where the total probability is
\begin{multline*}
\frac{(q^{\ell}t^{A})^{\Qc(a,c,d)}(1-t)}{1-q^{\ell}t^{A+1}}+\frac{t^{1-\Qc(a,c,d)}(1-q^{\ell}t^{A})}{1-q^{\ell}t^{A+1}}
=\begin{cases}
\frac{q^{\ell}t^{A}-q^{\ell}t^{A+1}+1-q^{\ell}t^{A}}{1-q^{\ell}t^{A+1}},& \Qc(a,c,d)=1\\
\frac{1-t+t-q^{\ell}t^{A+1}}{1-q^{\ell}t^{A+1}},& \Qc(a,c,d)=0
\end{cases}\quad=1,
\end{multline*}
where we write $A=\rarm(r,i+1)+1$ and $\ell:=\leg(r,i)+1=L-r+1$.
\end{proof}

\begin{example}\label{ex:prob tau}
Consider $\lambda=(4,4,2,2,1)$ and let 
\[
\sigma=\raisebox{15pt}{\tableau{4&1\\4&6\\3&6&2&1\\3&2&5&4&7}}\in\Tab(\lambda).
\] 
The $\lambda$-compatible indices are $1$ and $3$. The border is $\Top(\sigma)=4\,1\,|\,2\,1\,|\,7$ (the bar delimits the $\lambda$-comparable blocks), $\inc_{\lambda}(\Top(\sigma))=1\,4\,|\,1\,2\,|\,7$, and
\[
\Sym_{\lambda}(\Top(\sigma))=\{4\,1\,|\,2\,1\,|\,7,\ 1\,4\,|\,2\,1\,|\,7,\ 4\,1\,|\,1\,2\,|\,7,\ 1\,4\,|\,1\,2\,|\,7\}
\] 
with corresponding lengths (with respect to $\lambda$) equal to $2, 1, 1, 0$, respectively. 
$\widetilde{\op}_3(\sigma)$ produces the set $\{\T_3^{(2)}(\sigma)\}$ (case (ii)). We show the details for computing $\widetilde{\op}_1(\sigma)$: 
\begin{align*}
\widetilde{\op}_1(\sigma)=\widetilde{\op}^{(4)}_1(\sigma)&=\left\{\T_1^{(4)}(\sigma)\right\}\bigcup \widetilde{\op}_1^{(3)}(\T_1^{(4)}(\sigma))&\qquad \text{(case (iv))}\\
 \widetilde{\op}_1^{(3)}(\T_1^{(4)}(\sigma))&=\widetilde{\op}_1^{(2)}(\T_1^{[3,4]}(\sigma))&\qquad \text{(case (iii))}\\
  \widetilde{\op}_1^{(2)}(\T_1^{[3,4]}(\sigma))&=\left\{\T_1^{[3,4]}(\sigma)\right\}\bigcup \widetilde{\op}_1^{(1)}(\T_1^{[2,4]}(\sigma))&\qquad \text{(case (iv))}\\ 
  \widetilde{\op}_1^{(1)}(\T_1^{[2,4]}(\sigma))&=\left\{\T_1^{[1,4]}(\sigma))\right\}&\qquad \text{(case (i))}.
  \end{align*}
  Thus we get 
  \[
   \widetilde{\op}_1(\sigma)=\left\{\T_1^{(4)}(\sigma),\T_1^{[2,4]}(\sigma),\T_1^{[1,4]}(\sigma)\right\}=\left\{ \raisebox{15pt}{\tableau{1&4\\4&6\\3&6&2&1\\3&2&5&4&7}}\,,\ \raisebox{15pt}{\tableau{1&4\\6&4\\6&3&2&1\\3&2&5&4&7}}\,,\ \raisebox{15pt}{\tableau{1&4\\6&4\\6&3&2&1\\2&3&5&4&7}} \right\}.
   \]
   Labeling the tableaux in $\widetilde{\op}_1(\sigma)$ as $\sigma_1,\sigma_2,\sigma_3$, due to the fact that $\Qc(1,4,6)=0$ and $\Qc(6,3,2)=1$, 
    \begin{align*}
   \prob_1(\sigma,\sigma_1)&=\prob_1^{(4)}(\sigma,\sigma_1)=\frac{1-t}{1-qt^2},\\
   \prob_1(\sigma,\sigma_2)&=\prod_{r=2}^4\prob_1^{(r)}(\sigma,\sigma_2)=\frac{t(1-qt)}{1-qt^2}\cdot 1 \cdot \frac{q^3t^4(1-t)}{1-q^3t^5},\\
   \prob_1(\sigma,\sigma_2)&=\prod_{r=1}^4\prob_1^{(r)}(\sigma,\sigma_3)=\frac{t(1-qt)}{1-qt^2}\cdot 1 \cdot \frac{1-q^3t^4}{1-q^3t^5} \cdot 1.\\
   \end{align*}
   Adding these together, we confirm that $\prob_1(\sigma,\sigma_1)+\prob_1(\sigma,\sigma_2)+\prob_1(\sigma,\sigma_3)=1$. 
\end{example}

\begin{prop}\label{prop:balance}
Let $\lambda$ be a partition and let $1\leq i\leq \ell(\lambda)$ be $\lambda$-compatible. Let $\sigma\in\Tab(\lambda)$  with $\Top(\sigma)=w$, such that $w_i<w_{i+1}$, 
and let $\sigma' \in\widetilde{\op}_i(\sigma)$. Then $\Top(\sigma')=s_i\cdot w$ and
\begin{equation}\label{eq:prob balance}
\wt(\sigma')\prob(\sigma',\sigma)=t\wt(\sigma)\prob(\sigma,\sigma').
\end{equation}
\end{prop}

\begin{proof}
By definition, $\Top(\sigma')=s_i\cdot w$, since the swap $\T_i^{\lambda_i}$, which exchanges the topmost entries of columns $i$, $i+1$, is necessarily applied as part of $\widetilde{\op}$. 

It remains to show that \eqref{eq:prob balance} holds. Denote by $\wt^{(r)}(\sigma)$ the weight contribution of the pair of rows $r,r-1$ to $\wt(\sigma)$  for $2\leq r\leq \lambda_1$, not counting the $\quinv$ coming from any degenerate $L$-triples. Observe that the degenerate $L$-triples that contribute to $\coquinv(\sigma)$ are precisely the pairs of entries in $\Top(\sigma)$ that are comparable with respect to $\lambda$, and that form an inversion, and hence are counted by $\ell_{\lambda}(\Top(\sigma))$. Thus $t^{\ell_{\lambda}(\Top(\sigma))}$ is the contribution of the degenerate $L$-triples to $\quinv(\sigma)$. Then we may decompose $\wt(\sigma)$ as
\[
\wt(\sigma)=t^{\ell_{\lambda}(\Top(\sigma))}\prod_{r=2}^{\lambda_1}\wt^{(r)}(\sigma).
\] 
Note that \emph{Case (i)} is a specific instance of a degenerate triple when $\lambda_i=\lambda_{i+1}=1$. We shall compare $\wt^{(r)}(\sigma)$ to $\wt^{(r)}(\sigma')$ for $\sigma'\in \widetilde{\op}_i(\sigma)$ for $2\leq r\leq \lambda_i$. We base our computation on the following. First, we only need to consider the contribution to $\coinv$ from the pairs of cells $(r,i), (r-1,i)$ and $(r,i+1), (r-1,i+1)$. Second, we only need to consider the contribution to $\coquinv$ from the $L$-triples $(r,i), (r-1,i), (r-1,i+1)$ and $(r,i), (r-1,i), (r-1,j)$ where $j>i+1$ (i.e. the cell $(r-1,j)\in\Rarm(r,i+1)$). Finally, the arm and leg factor will only depend on the pairs of cells $(r,i), (r-1,i)$ and $(r,i+1), (r-1,i+1)$. Let $L:=\lambda_i$, $\ell:=\leg(r,i)+1=L-r+1$ and $A:=\rarm(r,i+1)+1$. For $r\geq 2$, we have the following cases, based on the entries $a=\sigma(r,i), b=\sigma(r,i+1), c=\sigma(r-1,i)$, and $d=\sigma(r-1,i+1)$. 

\noindent\emph{Case (ii):} $\sigma'=\T_i^{[r,L]}(\sigma)$. If $\{a,b\}\cap \{c,d\}=\emptyset$, we have the following configuration when restricted to the rows $r-1,r$ and columns $i,i+1$:
\[
\raisebox{-15pt}{\begin{tikzpicture}[scale=0.4]\cell00{$c$}\cell01{$d$}\cell{-1}0{$a$}\cell{-1}1{$b$}\node at (-1.5,1.5) {\tiniest{$r$}};\node at (-2,.5) {\tiniest{$r-1$}};\node at (-.5,-.5) {\tiniest{$i$}};\node at (.5,-.5) {\tiniest{$i+1$}};\end{tikzpicture}}\quad\xrightarrow{\T_i^{(r)}}\quad\raisebox{-15pt}{\begin{tikzpicture}[scale=0.4]\cell00{$c$}\cell01{$d$}\cell{-1}0{$b$}\cell{-1}1{$a$}\node at (-.5,-.5) {\tiniest{$i$}};\node at (.5,-.5) {\tiniest{$i+1$}};\end{tikzpicture}}
\]

The arm and leg factors are identical in the two tableaux. We analyze the contribution to $\wt^{(r)}(\sigma)$ of these four cells. Note that $\Qc(a,c,d)=\Qc(b,c,d)$ implies that the cyclic orders $a<c<b<d$ and $a<d<b<c$ cannot occur. This yields the five possible (non-cyclic) orderings shown in the columns of the table below (since $\T_i^{(r)}$ is an involution, we assume without loss of generality that $a<b$).

\begin{center}
\begin{tabular}{l||c|c|c|c|c}
order&$a,b<c,d$&$a<c,d<b$&$c<a,b<d$&$d<a,b<c$&$c,d<a,b$\\\hline
$\wt^{(r)}(\sigma)$&$1$&$q^{\ell}$&$q^{\ell}$&$q^{\ell}$&$q^{2\ell}$\\\hline
$\wt^{(r)}(\sigma')$&$1$&$q^{\ell}$&$q^{\ell}$&$q^{\ell}$&$q^{2\ell}$\\\hline
ratio&$1$&$1$&$1$&$1$&$1$
\end{tabular}
\end{center}

Next we compare the respective contributions to $\wt^{(r)}(\sigma)$ coming from all $L$-triples involving the four cells together with some cell $x=(r-1,j)\in\Rarm(r,i+1)$ (with $j>i+1$), and we denote $f=\sigma(x)$, as in the configuration below, restricted to the rows $r-1,r$ and columns $i,i+1,j$. 
\[
\raisebox{-15pt}{\begin{tikzpicture}[scale=0.4]\cell00{$c$}\cell01{$d$}\cell{-1}0{$a$}\cell{-1}1{$b$}\cell{0}4{$f$}\node at (2,.5) {$\cdots$};\node at (-1.5,1.5) {\tiniest{$r$}};\node at (-2,.5) {\tiniest{$r-1$}};\node at (-.5,-.5) {\tiniest{$i$}};\node at (.5,-.5) {\tiniest{$i+1$}};\node at (3.5,-.5) {\tiniest{$j$}};\end{tikzpicture}}\quad\xrightarrow{\T_i^{(r)}}\quad\raisebox{-15pt}{\begin{tikzpicture}[scale=0.4]\cell00{$c$}\cell01{$d$}\cell{-1}0{$b$}\cell{-1}1{$a$}\cell{0}4{$f$}\node at (2,.5) {$\cdots$};\node at (-.5,-.5) {\tiniest{$i$}};\node at (.5,-.5) {\tiniest{$i+1$}};\node at (3.5,-.5) {\tiniest{$j$}};\end{tikzpicture}}
\]

Since cyclic order is sufficient to determine whether a $L$-triple is a $\quinv$ triple, we organize the table based on the relative order of $f$ with respect to $a,b,c,d$, which are cyclically ordered according to $\Qc(a,c,d)=\Qc(b,c,d)$. We represent the cyclic order of a set of distinct entries by writing them in a circle that is intended to be read clockwise, making the following correspondence:
\[
\raisebox{-15pt}{\begin{tikzpicture}\circleU cba\end{tikzpicture}} \longrightarrow\ \Qc(a,b,c)=0\qquad\qquad\qquad \raisebox{-15pt}{\begin{tikzpicture}\circleU bca\end{tikzpicture}}\longrightarrow\ \Qc(a,b,c)=1
\]

\begin{table}[H]
\caption{\ }
\begin{center}
\begin{tabular}{l||c|c|c||c|c}
order&$\begin{tikzpicture}\circleU {{a,b}}{{c,d}}f\end{tikzpicture}$&$\begin{tikzpicture}\circleU {{a,b}}f{{c,d}}\end{tikzpicture}$&$\begin{tikzpicture}\circleS {{a,b}}cfd\end{tikzpicture}$&$\begin{tikzpicture}\circleS {{a,b}}dfc\end{tikzpicture}$&$\begin{tikzpicture}\circleS afb{{c,d}}\end{tikzpicture}$
\\\hline
$\wt^{(r)}(\sigma)$&$t^2$&$1$&$t$&$t$&$t$\\\hline
$\wt^{(r)}(\sigma')$&$t^2$&$1$&$t$&$t$&$t$\\\hline
ratio&$1$&$1$&$1$&$1$&$1$
\end{tabular}
\end{center}
\label{table:ii}
\end{table}

Since this is true for all $x\in\Rarm(r,i+1)$, we have $\wt^{(r)}(\sigma)=\wt^{(r)}(\sigma')$. 

Next we analyze the case where $b=c$. This is the following configuration when restricted to the rows $r-1,r$ and columns $i,i+1$:
\[
\raisebox{-15pt}{\begin{tikzpicture}[scale=0.4]\cell00{$b$}\cell01{$d$}\cell{-1}0{$a$}\cell{-1}1{$b$}\node at (-1.5,1.5) {\tiniest{$r$}};\node at (-2,.5) {\tiniest{$r-1$}};\node at (-.5,-.5) {\tiniest{$i$}};\node at (.5,-.5) {\tiniest{$i+1$}};\end{tikzpicture}}\quad\xrightarrow{\T_i^{(r)}}\quad\raisebox{-15pt}{\begin{tikzpicture}[scale=0.4]\cell00{$b$}\cell01{$d$}\cell{-1}0{$b$}\cell{-1}1{$a$}\node at (-.5,-.5) {\tiniest{$i$}};\node at (.5,-.5) {\tiniest{$i+1$}};\end{tikzpicture}}
\]
We compare the contribution to $\wt^{(r)}(\sigma)$ versus $\wt^{(r)}(\sigma')$ coming from these four cells in the table below, based on the relative orders of $a,b,d$, organized based on whether or not $\Qc(a,b,d)=0$:

\begin{center}
\begin{tabular}{l||c|c|c||c|c|c}
&\multicolumn{3}{c}{$\Qc(a,b,d)=0$}&\multicolumn{3}{c}{$\Qc(a,b,d)=1$}\\[.2cm]
order&$a<b<d$&$b<d<a$&$d<a<b$&$a<d<b$&$d<b<a$&$b<a<d$\\\hline
$\wt^{(r)}(\sigma)$&$1$&$q^{\ell}$&$q^{\ell}$&$q^{\ell}t$&$q^{2\ell}t$&$q^{\ell}t$\\\hline
$\wt^{(r)}(\sigma')$&$1$&$q^{\ell}$&$q^{\ell}$&$1$&$q^{\ell}$&$1$\\\hline
ratio&$1$&$1$&$1$&$q^{\ell}t$&$q^{\ell}t$&$q^{\ell}t$
\end{tabular}
\end{center}
Thus the ratio of the respective contributions is $1$ if $\Qc(a,b,d)=0$, and $q^{\ell}t$ otherwise.

Next we compare the respective contributions to $\wt^{(r)}(\sigma)$ coming from all $L$-triples involving the four cells together with some cell $x=(r-1,j)\in\Rarm(r,i+1)$ with $j>i+1$, and we denote $f=\sigma(x)$, as in the configuration below restricted to the rows $r-1,r$ and columns $i,i+1,j$. 
\[
\raisebox{-15pt}{\begin{tikzpicture}[scale=0.4]\cell00{$b$}\cell01{$d$}\cell{-1}0{$a$}\cell{-1}1{$b$}\cell{0}4{$f$}\node at (2,.5) {$\cdots$};\node at (-1.5,1.5) {\tiniest{$r$}};\node at (-2,.5) {\tiniest{$r-1$}};\node at (-.5,-.5) {\tiniest{$i$}};\node at (.5,-.5) {\tiniest{$i+1$}};\node at (3.5,-.5) {\tiniest{$j$}};\end{tikzpicture}}\quad\xrightarrow{\T_i^{(r)}}\quad\raisebox{-15pt}{\begin{tikzpicture}[scale=0.4]\cell00{$b$}\cell01{$d$}\cell{-1}0{$b$}\cell{-1}1{$a$}\cell{0}4{$f$}\node at (2,.5) {$\cdots$};\node at (-.5,-.5) {\tiniest{$i$}};\node at (.5,-.5) {\tiniest{$i+1$}};\node at (3.5,-.5) {\tiniest{$j$}};\end{tikzpicture}}
\]

Since cyclic order is sufficient to determine whether a $L$-triple is a $\quinv$ triple, we organize the table based on the relative order of $f$ with respect to $a,b,d$, and based on whether or not $\Qc(a,b,d)=0$.

\begin{center}
\begin{tabular}{l||c|c|c||c|c|c}
&\multicolumn{3}{c}{$\Qc(a,b,d)=0$}&\multicolumn{3}{c}{$\Qc(a,b,d)=1$}\\[.2cm]
order&$\begin{tikzpicture}\circleS dbaf\end{tikzpicture}$&$\begin{tikzpicture}\circleS dbfa\end{tikzpicture}$&$\begin{tikzpicture}\circleS dfba\end{tikzpicture}$&$\begin{tikzpicture}\circleS dabf\end{tikzpicture}$&$\begin{tikzpicture}\circleS dafb\end{tikzpicture}$&$\begin{tikzpicture}\circleS dfab\end{tikzpicture}$\\\hline
$\wt^{(r)}(\sigma)$&$1$&$t$&$t$&$t$&$t$&$t^2$\\\hline
$\wt^{(r)}(\sigma')$&$1$&$t$&$t$&$1$&$1$&$t$\\\hline
ratio&$1$&$1$&$1$&$t$&$t$&$t$
\end{tabular}
\end{center}

Since the above holds true for any cell $x\in\Rarm(r+1,i)$, from the two tables above we have that 
\[\frac{\wt^{(r)}(\sigma)}{\wt^{(r)}(\sigma')}=\frac{1-t}{1-q^{\ell}t^{A+1}}\times \begin{cases} 1,&\Qc(a,b,d)=0\\
q^{\ell}t^{A},&\Qc(a,b,d)=1.
\end{cases}
\]
Note that the pre-factor comes from the fact that $\sigma(r+1,i)\neq \sigma(r,i)$, whereas $\sigma'(r+1,i)= \sigma'(r,i)$. 

 \noindent\emph{Case (iii):} $\sigma'=\widetilde{\tau}_i^{(r-1)}(\T_i^{[r,L]}(\sigma))$. If $\{a,b\}\cap \{c,d\}=\emptyset$, we have the following configuration when restricted to the rows $r-1,r$ and columns $i,i+1$:
\[
\raisebox{-15pt}{\begin{tikzpicture}[scale=0.4]\cell00{$c$}\cell01{$d$}\cell{-1}0{$a$}\cell{-1}1{$b$}\node at (-1.5,1.5) {\tiniest{$r$}};\node at (-2,.5) {\tiniest{$r-1$}};\node at (-.5,-.5) {\tiniest{$i$}};\node at (.5,-.5) {\tiniest{$i+1$}};\end{tikzpicture}}\quad\xrightarrow{\T_i^{(r-1)}\circ\T_i^{(r)}}\quad\raisebox{-15pt}{\begin{tikzpicture}[scale=0.4]\cell00{$d$}\cell01{$c$}\cell{-1}0{$b$}\cell{-1}1{$a$}\node at (-.5,-.5) {\tiniest{$i$}};\node at (.5,-.5) {\tiniest{$i+1$}};\end{tikzpicture}}
\]

Both $\coinv$ and the arm and leg factors are identical between rows $r,r-1$ in $\sigma'$ and $\sigma$. Moreover, all $L$-triples are preserved with the exception of the triple $(b,d,c)$ replacing the triple $(a,c,d)$. Then we have $\Qc(b,d,c)=1-\Qc(b,c,d)=1-(1-\Qc(a,c,d))=\Qc(a,c,d)$. Thus $\wt^{(r)}(\sigma')=\wt^{(r)}(\sigma)$ in this case.
 
 Now we analyze the case where $b=d$, which is the following configuration when restricted to the rows $r-1,r$ and columns $i,i+1$:
\[
\raisebox{-15pt}{\begin{tikzpicture}[scale=0.4]\cell00{$c$}\cell01{$b$}\cell{-1}0{$a$}\cell{-1}1{$b$}\node at (-1.5,1.5) {\tiniest{$r$}};\node at (-2,.5) {\tiniest{$r-1$}};\node at (-.5,-.5) {\tiniest{$i$}};\node at (.5,-.5) {\tiniest{$i+1$}};\end{tikzpicture}}\quad\xrightarrow{\T_i^{(r-1)}\circ\T_i^{(r)}}\quad\raisebox{-15pt}{\begin{tikzpicture}[scale=0.4]\cell00{$b$}\cell01{$c$}\cell{-1}0{$b$}\cell{-1}1{$a$}\node at (-.5,-.5) {\tiniest{$i$}};\node at (.5,-.5) {\tiniest{$i+1$}};\end{tikzpicture}}
\]
Since all $L$-triples are preserved except for the triple $(a,c,b)$ in $\sigma$ coming from the four cells above, and since $\sigma(r+1,i)= \sigma(r+1,i)$, whereas $\sigma'(r,i)= \sigma'(r,i)$, we have 
\[\frac{\wt^{(r)}(\sigma)}{\wt^{(r)}(\widetilde{\op}_i^{(r)}(\sigma))}=\frac{1-q^{\ell}t^{A}}{1-q^{\ell}t^{A+1}}\times \begin{cases} 1,&\Qc(a,c,b)=0\\ t,&\Qc(a,c,b)=1.\end{cases}
\]
 
  \noindent\emph{Case (iv).} If $a=c$, let $\sigma_1=\T_i^{(r)}(\sigma)$ and $\sigma_2=\widetilde{\op}_i^{(r-1)}(\T_i^{(r)}(\sigma))$. Restricted to the rows $r-1,r$ and columns $i,i+1$, we have the following configurations:
\[
\sigma=\tableau{a&b\\a&d}\ ,\qquad \sigma_1=\tableau{b&a\\a&d}\ ,\qquad \sigma_2=\tableau{b&a\\d&a}\ .
\]
From our analysis of cases (ii) and (iii), we have that
\begin{align*}
\wt^{(r)}(\sigma_1)&=\frac{1-t}{1-q^{\ell}t^{A+1}}\wt^{(r)}(\sigma)\times\begin{cases}1,&\Qc(b,a,d)=0\\
q^{\ell}t^{A},&\Qc(b,a,d)=1 \end{cases}&=\wt^{(r)}(\sigma)\prob(\sigma, \sigma_1),\\
\wt^{(r)}(\sigma_2)&=\frac{1-q^{\ell}t^{A}}{1-q^{\ell}t^{A+1}}\wt^{(r)}(\sigma)\times\begin{cases}1,&\Qc(b,d,a)=0\\
t,&\Qc(b,d,a)=1 \end{cases}&=\wt^{(r)}(\sigma)\prob(\sigma, \sigma_2).
\end{align*}
Moreover, since $\Qc(b,a,d)=1-\Qc(b,d,a)$, we have that $\wt^{(r)}(\sigma_1)\prob(\sigma_1, \sigma)+\wt^{(r)}(\sigma_2)\prob(\sigma_2, \sigma)$ is equal to:
\[
\wt^{(r)}(\sigma)\times\begin{cases}
\frac{1-t}{1-q^{\ell}t^{A+1}}+\frac{t(1-q^{\ell}t^{A})}{1-q^{\ell}t^{A+1}},&\Qc(b,a,d)=0\\
\frac{q^{\ell}t^{A}(1-t)}{1-q^{\ell}t^{A+1}}+\frac{1-q^{\ell}t^{A}}{1-q^{\ell}t^{A+1}},&\Qc(b,a,d)=1
\end{cases}=\wt^{(r)}(\sigma).
\]

For all cases, if $r=L$, then $\ell_{\lambda}(\Top(\sigma'))-\ell_{\lambda}(\Top(\sigma))=\ell_{\lambda}(w')-\ell_{\lambda}(w)=1$, and otherwise there is no change to the top border of the tableau.

Suppose $\sigma'\in\widetilde{\op}_i(\sigma)$ such that $\sigma'=\T_i^{[k,L]}(\sigma)$ so that $\prob(\sigma, 
\sigma')=\prod_{r=k}^L \prob_i^{(r)}(\sigma,\sigma')$. In particular, $\sigma$ and $\sigma'$ are identical on rows $k-1$ and below. Then 
\begin{align*}
\wt(\sigma')\prob(\sigma',\sigma)&=\left(t^{\ell_{\lambda}(w')}\prod_{r=2}^L \wt^{(r)}(\sigma')\right)\prob_i(\sigma', \sigma)\\
&= t^{\ell_{\lambda}(w)+1}\left(\prod_{r=k}^L \wt^{(r)}(\widetilde{\op}_i^{(r)}(\sigma))\prob_i^{(r)}(\sigma',\sigma))\right)\left(\prod_{r=2}^{k-1} \wt^{(r)}(\sigma)\right)\\
&=t^{\ell_{\lambda}(w)+1}\left(\prod_{r=k}^L \wt^{(r)}(\sigma)\prob_i^{(r)}(\sigma,\sigma')\right)\left(\prod_{r=2}^{k-1} \wt^{(r)}(\sigma)\right)\\
&=t\wt(\sigma)\prob(\sigma,\sigma').
\end{align*}
\end{proof}

\begin{example}\label{ex:Prop 3.5} We show an example of \cref{prop:balance} for the tableau $\sigma$ from \cref{ex:prob tau}. Let $\widetilde{\op}_1(\sigma)=\{\sigma_1,\sigma_2,\sigma_3\}$. Their weights are
   \begin{align*}
   \wt(\sigma)&=\frac{q^5t^7(1-t)^5}{(1-qt)(1-q^2t^4)(1-q^3t^4)(1-qt^3)(1-qt^2)},\\
    \wt(\sigma_1)&=\frac{q^5t^6(1-t)^6}{(1-qt)(1-q^2t^4)(1-q^3t^4)(1-qt^3)(1-qt^2)^2},\\
     \wt(\sigma_2)&=\frac{q^8t^{10}(1-t)^6}{(1-q^2t^3)(1-q^3t^5)(1-q^3t^4)(1-qt^3)(1-qt^2)^2},\\
     \wt(\sigma_3)&=\frac{q^5t^{6}(1-t)^5}{(1-q^2t^3)(1-q^3t^5)(1-qt^3)(1-qt^2)^2}.  
   \end{align*}
and the probabilities $\prob(\sigma,\sigma_i)$ (computed explicitly in \cref{ex:tau}) and $\prob(\sigma_i,\sigma)$ are
\begin{align*}
 \prob(\sigma,\sigma_1)&=\frac{1-t}{1-qt^2},&\quad \prob(\sigma,\sigma_2)&=\frac{q^3t^5(1-t)(1-qt)}{(1-qt^2)(1-q^3t^5)},&\quad  \prob(\sigma,\sigma_2)&=\frac{t(1-qt)(1-q^3t^4)}{(1-qt^2)(1-q^3t^5)},\\
 \prob(\sigma_1,\sigma)&=1,&\quad  \prob(\sigma_2,\sigma)&=\frac{t(1-q^2t^3)}{1-q^2t^4}, &\quad  \prob(\sigma_3,\sigma)&=\frac{t(1-q^2t^3)}{1-q^2t^4}.
\end{align*}
In $w=\Top(\sigma)=4\,1\,|\,2\,1\,|\,7$, we have $w_1>w_2$. We check that for $i=1,2,3$, the quantities above satisfy
\[
\wt(\sigma)\prob(\sigma,\sigma_i)=t\wt(\sigma_i)\prob(\sigma_i,\sigma).
\]
\end{example}

It follows that the weight of a tableau $\sigma$ can be computed as a weighted sum over all tableaux in the set $\widetilde{\op}_i(\sigma)$.
\begin{lemma}\label{lem:sum tau}
Let $\sigma\in\Tab(\lambda)$ be a $\quinv$-non-attacking tableau with $\Top(\sigma)=w$. Let $1\leq i\leq \ell(\lambda)$ such that $\lambda_i=\lambda_{i+1}$. Then
\begin{equation}\label{eq:balance}
\wt(\sigma)=\sum_{\substack{\sigma'\in\Tab(\lambda)\\\prob_i(\sigma, \sigma')\neq 0}} \wt(\sigma')\prob_i(\sigma',\sigma)\times\begin{cases} t,& w_i>w_{i+1}\\t^{-1},& w_i<w_{i+1}.\end{cases}
\end{equation}
\end{lemma}

\begin{proof} Without loss of generality, assume $w_i>w_{i+1}$.Summing over all $\sigma'$ such that $\prob_i(\sigma, \sigma')\neq 0$:
\[
\wt(\sigma)=\wt(\sigma)\sum_{\sigma'\in\Tab(\lambda)}\prob_i(\sigma,\sigma')=\sum_{\sigma'\in\Tab(\lambda)}t\wt(\sigma')\prob_i(\sigma',\sigma),
\]
where the second equality is due to \cref{prop:balance}. 
\end{proof}

\begin{defn}\label{def:prob sequence}
Let $\sigma\in\Tab(\lambda)$ be a $\quinv$-non-attacking tableau with $\Top(\sigma)=w$. Let $s_{i_{\ell}}\cdots s_{i_1}\in\PDS(\lambda)$ be the PDS such that $s_{i_{\ell}}\cdots s_{i_1} \cdot w=\inc_{\lambda}(w)$. (Note that $s_{i_j}$ is by definition $\lambda$-compatible for each $j$). Suppose $\sigma \in\widetilde{\op}_{\ell}\circ \cdots\circ\widetilde{\op}_1(\sigma')$ for some $\sigma'$ with $\Top(\sigma')=\inc_{\lambda}(w)$ (i.e $\sigma'$ is $\coquinv$-sorted). Then we write the probability of producing $\sigma$ from $\sigma'$ as the sum over the probabilities of all possible chains $(\sigma'_0,\sigma'_1,\ldots,\sigma'_{\ell})$ of tableaux that start with $\sigma'_0=\sigma'$ and end with $\sigma'_{\ell}=\sigma$, and the $k$'th element of the chain is produced by applying $\widetilde{\op}_{i_k}$ to the $k-1$'st element of the chain, for $1\leq k\leq \ell$:
\begin{equation}\label{eq:totalprob}
\prob(\sigma',\sigma)= 
\sum_{\substack{(\sigma'_{\ell-1},\ldots,\sigma'_1,\sigma'_0=\sigma')\in\Tab(\lambda)^{\ell}\\\sigma'_{k}\in\widetilde{\op}_{i_{k}}(\sigma'_{k-1}),\ 1\leq k\leq\ell-1\\
\sigma \in\widetilde{\op}_{i_{\ell}}(\sigma'_{\ell-1})}}\ 
 \prod_{k=1}^{\ell} \prob_{i_k}(\sigma'_{k-1},\sigma'_{k}).
\end{equation}
\end{defn}

In particular, by iterating \cref{lem:welldefined}, we get that for any $\coquinv$-sorted tableau $\tau$ with top border $v$ such that $\ell_{\lambda}(v)=0$ and any $\lambda$-compatible rearrangement $w\in\Sym_{\lambda}(v)$,
\begin{equation}\label{eq:sum is 1}
\sum_{\substack{\sigma'\in\Tab(\lambda)\\\Top(\sigma')=w}}\prob(\tau,\sigma')=1.
\end{equation}

Note that the definition of $\prob(\sigma,\sigma')$ can be extended to any pair of tableaux $\sigma,\sigma'\in\Tab(\lambda)$, as the sum over all admissible chains of tableaux arising from the application of a sequence of operators given by a PDS to get from $\Top(\sigma)$ to $\Top(\sigma')$. When no such PDS exists, or when there are no such admissible chains of tableaux, this probability is zero, as expected. For our purposes, it suffices to limit the definition to $\coquinv$-sorted tableaux $\sigma$.

\begin{lemma}\label{lem:perm}
Let $\sigma\in\Tab(\lambda)$ be a $\quinv$-non-attacking tableau with $\Top(\sigma)=w$. Then
\[
\wt(\sigma)=t^{\ell_{\lambda}(w)}\sum_{\substack{\sigma'\in\Tab(\lambda)\\\Top(\sigma')=\inc_{\lambda}(w)}} \wt(\sigma')\prob(\sigma',\sigma).
\]
\end{lemma}
\begin{proof} 
Let $s_{i_{\ell}}\cdots s_{i_1}$ be the PDS such that $s_{i_{\ell}}\cdots s_{i_1}\cdot w=\inc_{\lambda}(w)$. Note that $s_{i_j}$ is by definition $\lambda$-compatible for each $j$, and that $\ell_{\lambda}(w)=\ell$. Repeated application of \cref{lem:sum tau} yields 
\begin{equation}\label{eq:totalwt}
\wt(\sigma)=t^{\ell}\sum_{\substack{\sigma'\in\Tab(\lambda)\\\Top(\sigma')=\inc_{\lambda}(w)}} \wt(\sigma')
\sum_{\substack{(\sigma'_{\ell-1},\ldots,\sigma'_1,\sigma'_0=\sigma')\in\Tab(\lambda)^{\ell}\\\sigma'_{j}\in\widetilde{\op}_{i_{j}}(\sigma'_{j-1}),\ 1\leq j\leq\ell-1\\
\sigma \in\widetilde{\op}_{i_{\ell}}(\sigma'_{\ell-1})}}\ 
 \prod_{j=1}^{\ell} \prob_{i_j}(\sigma'_{j-1},\sigma'_{j}).
\end{equation}
The claim follows from \cref{def:prob sequence}.
\end{proof}

We now have all the ingredients to prove our main result, \cref{thm:main}.

\begin{proof}[Proof of \cref{thm:main}] Let $\lambda$ be a partition with $k:=\ell(\lambda)$. The right hand side of \eqref{eq:P quinv} can be written as a sum over all possible borders $w\in\Z^{k}_{>0}$ of $\quinv$-non-attacking fillings in $\Tab(\lambda)$: 
\begin{align}
\perm_{\lambda}(t)P_{\lambda}(X;q,t)&=\sum_{\substack{\sigma\in\Tab(\lambda)}}\wt(\sigma)=\sum_{w\in\Z_{>0}^{k}}\ \sum_{\substack{\sigma\in\Tab(\lambda)\\\Top(\sigma)=w}}\wt(\sigma)\nonumber\\
&=\sum_{w\in\Z_{>0}^{k}}\ \sum_{\substack{\sigma\in\Tab(\lambda)\\\Top(\sigma)=w}}\ \sum_{\substack{\sigma'\in\Tab(\lambda)\\\Top(\sigma')=\inc_{\lambda}(w)}} t^{\ell_{\lambda}(w)}\wt(\sigma')\prob(\sigma',\sigma)\label{eq:f1}\\
&=\sum_{w\in\Z_{>0}^{k}}t^{\ell_{\lambda}(w)} \sum_{\substack{\sigma'\in\Tab(\lambda)\\\Top(\sigma')=\inc_{\lambda}(w)}} \wt(\sigma')=\sum_{\substack{v\in\Z_{>0}^{k}\\\ell_{\lambda}(v)=0}}\ \sum_{\substack{\sigma'\in\Tab(\lambda)\\\Top(\sigma')=v}}\wt(\sigma')\sum_{w\in\Sym_{\lambda}(v)}t^{\ell_{\lambda}(w)}\label{eq:f2}\\
&=\sum_{\substack{v\in\Z_{>0}^{k}\\\ell_{\lambda}(v)=0}}\ \sum_{\substack{\sigma'\in\Tab(\lambda)\\\Top(\sigma')=v}}\wt(\sigma')\perm_{\lambda}(t)=\perm_{\lambda}(t) \sum_{\substack{\sigma\in\Tab(\lambda)\\\sigma\,\cqsort}} \wt(\sigma). \label{eq:f3}
\end{align}
where \eqref{eq:f1} is by \cref{lem:perm}, \eqref{eq:f2} is by \eqref{eq:sum is 1}, \eqref{eq:f3} is by \eqref{eq:perm}, and the final equality is the definition of a $\quinv$-sorted filling. The result is obtained by cancelling out the $\perm_{\lambda}(t)$ factor.
\end{proof}

\section{A compact \texttt{inv} formula for $P_{\lambda}(X;q,t)$}\label{sec:inv operators}

A similar strategy can be used to compress \eqref{eq:P inv} using probabilistic operators that extend the $\tau_i$'s to the non-attacking setting. Denote the set of  $\inv$-non-attacking fillings $\sigma:\dg(\lambda)\rightarrow\Z^+$ by $\Tab'(\lambda)$. For $\sigma\in\Tab'(\lambda)$, define
\[
\wt_{\HHL}(\sigma)=t^{\coinv(\sigma)}q^{\maj(\sigma)}\prod_{\substack{u,~\South(u)\in\dg(\lambda)\\\sigma(u)\neq \sigma(\South(u))}}\frac{1-t}{1-q^{\leg(u)}t^{\arm(u)}},
\]
so that \eqref{eq:P inv} becomes $P_{\lambda}(X;q,t)=\sum_{\substack{\sigma\in\Tab'(\lambda)}}\wt_{\HHL}(\sigma)$. Then we obtain a new compact formula for $P_{\lambda}$ in terms of the  $\inv$ statistic on $\coinv$-sorted  $\inv$-non-attacking tableaux (see \cref{def:sorted}).
\begin{theorem}\label{thm:compact inv}
The Macdonald polynomial $P_{\lambda}(X;q,t)$ is given by
\begin{equation}\label{eq:p1}
P_{\lambda}(X;q,t)=\Pi_{\lambda}(q,t)\hspace{-0.1in}\sum_{\substack{\sigma\in\Tab'(\lambda)\\\sigma\,\cisort}}\hspace{-0.1in}\wt_{\HHL}(\sigma),
\end{equation}
where 
\[
\Pi_{\lambda}(q,t)=\prod_{u\in\overline{\dg}(\lambda)}\frac{1-q^{\leg(u)+1}t^{\arm(u)+1}}{1-q^{\leg(u)+1}t^{\rarm(u)+1}}.
\]
\end{theorem}

The proof of \cref{thm:compact inv} is identical to the proof of its $\quinv$ analog \cref{thm:main} modulo the definition of the probabilistic operators. Define the probabilistic  $\inv$ operators, denoted $\widetilde{\tau}_i$, as follows.
\begin{defn}
Let $\sigma\in\Tab'(\lambda,n)$ be an  $\inv$-non-attacking tableau, and let $1\leq i\leq \ell(\lambda)$ be $\lambda$-compatible with $L=\lambda_i=\lambda_{i+1}$. Define the operator $\widetilde{\tau}_i(\sigma):=\widetilde{\tau}_i^{(1)}(\sigma)$, where for $1\leq r\leq \lambda_i$,  $\widetilde{\tau}_i^{(r)}$ acts on $\sigma$ by probabilistically producing a set of tableaux, denoted by $\widetilde{\tau}_i^{(r)}(\sigma)$, based on the following cases for the entries $a=\sigma(r,i)$, $b=\sigma(r,i+1)$, $c=\sigma(r+1,i)$, and $d=\sigma(r+1,i+1)$ (assume all configurations are $\inv$-non-attacking):

\begin{center}
(i) \raisebox{-20pt}{\begin{tikzpicture}[scale=0.4]\cell{-1}0{}\cell{-1}1{}
\node at (-.5,-1.6) {$a$};\node at (.5,-1.5) {$b$};
\node at (-.5,-.5) {$\emptyset$};\node at (.5,-.5) {$\emptyset$};
\node at (-1.5,-1.5) {\tiniest{$L$}};\node at (-1.5,-.5) {};
\node at (-.6,-2.5) {\tiniest{$i$}};\node at (.6,-2.5) {\tiniest{$i+1$}};
\draw[->] (0,.1)--(0,.9);
\cell{2}0{}\cell{2}1{}
\node at (-.5,1.5) {$b$};\node at (.5,1.4) {$a$};
\node at (-.5,2.5) {$\emptyset$};\node at (.5,2.5) {$\emptyset$};
\node at (-1.5,1.5) {\tiniest{$L$}};\node at (-1.5,2.5) {};
\draw[blue] (-.9,1.9) rectangle (.9,1.1);
\end{tikzpicture}}
\quad(ii) \raisebox{-20pt}{\begin{tikzpicture}[scale=0.4]\cell{-1}0{}\cell{-1}1{}\cell{-2}0{}\cell{-2}1{}
\node at (-.5,-1.6) {$a$};\node at (.5,-1.5) {$b$};
\node at (-.5,-.5) {$b$};\node at (.5,-.5) {$d$};
\node at (-1.5,-1.5) {\tiniest{$r$}};\node at (-2,-.5) {\tiniest{$r+1$}};
\node at (-.6,-2.5) {\tiniest{$i$}};\node at (.6,-2.5) {\tiniest{$i+1$}};
\draw[->] (0,.1)--(0,.9);
\cell{2}0{}\cell{2}1{}\cell{1}0{}\cell{1}1{}
\node at (-.5,1.5) {$b$};\node at (.5,1.4) {$a$};
\node at (-.5,2.5) {$b$};\node at (.5,2.5) {$d$};
\node at (-1.5,1.5) {\tiniest{$r$}};\node at (-2,2.5) {\tiniest{$r+1$}};
\draw[blue] (-.9,1.9) rectangle (.9,1.1);
\end{tikzpicture}} or\ \ 
\raisebox{-20pt}{\begin{tikzpicture}[scale=0.4]\cell{-1}0{}\cell{-1}1{}\cell{-2}0{}\cell{-2}1{}
\node at (-.5,-1.6) {$a$};\node at (.5,-1.5) {$b$};
\node at (-.5,-.6) {$c$};\node at (.5,-.5) {$d$};
\node at (-.6,-2.5) {\tiniest{$i$}};\node at (.6,-2.5) {\tiniest{$i+1$}};
\draw[->] (0,.1)--(0,.9);
\cell{2}0{}\cell{2}1{}\cell{1}0{}\cell{1}1{}
\node at (-.5,1.5) {$b$};\node at (.5,1.4) {$a$};
\node at (-.5,2.4) {$c$};\node at (.5,2.5) {$d$};
\draw[blue] (-.9,1.9) rectangle (.9,1.1);
\end{tikzpicture}}
\quad (iii) \raisebox{-20pt}{\begin{tikzpicture}[scale=0.4]\cell{-1}0{}\cell{-1}1{}\cell{-2}0{}\cell{-2}1{}
\node at (-.5,-1.6) {$a$};\node at (.5,-1.5) {$b$};
\node at (-.5,-.5) {$c$};\node at (.5,-.5) {$b$};
\node at (-1.5,-1.5) {\tiniest{$r$}};\node at (-2,-.5) {\tiniest{$r+1$}};
\node at (-.6,-2.5) {\tiniest{$i$}};\node at (.6,-2.5) {\tiniest{$i+1$}};
\draw[->] (0,.1)--(0,.9);
\cell{2}0{}\cell{2}1{}\cell{1}0{}\cell{1}1{}
\node at (-.5,1.5) {$b$};\node at (.5,1.4) {$a$};
\node at (-.5,2.5) {$b$};\node at (.5,2.5) {$c$};
\node at (-1.5,1.5) {\tiniest{$r$}};\node at (-2,2.5) {\tiniest{$r+1$}};
\draw[blue] (-.9,1.9) rectangle (.9,1.1);
\draw[blue] (-.9,2.9) rectangle (.9,2.1);
\end{tikzpicture}} or\ \ 
\raisebox{-20pt}{\begin{tikzpicture}[scale=0.4]\cell{-1}0{}\cell{-1}1{}\cell{-2}0{}\cell{-2}1{}
\node at (-.5,-1.6) {$a$};\node at (.5,-1.5) {$b$};
\node at (-.5,-.6) {$c$};\node at (.5,-.5) {$d$};
\node at (-.6,-2.5) {\tiniest{$i$}};\node at (.6,-2.5) {\tiniest{$i+1$}};
\draw[->] (0,.1)--(0,.9);
\cell{2}0{}\cell{2}1{}\cell{1}0{}\cell{1}1{}
\node at (-.5,1.5) {$b$};\node at (.5,1.4) {$a$};
\node at (-.5,2.4) {$c$};\node at (.5,2.5) {$d$};
\draw[blue] (-.9,1.9) rectangle (.9,1.1);
\draw[blue] (-.9,2.9) rectangle (.9,2.1);
\end{tikzpicture}}
\quad (iv)\hspace{-0.3in} \raisebox{-20pt}{\begin{tikzpicture}[scale=0.4]
\cell{-1}2{}\cell{-1}3{}\cell{-2}2{}\cell{-2}3{}
\cell{-1}{-1}{}\cell{-1}{-2}{}\cell{-2}{-1}{}\cell{-2}{-2}{}
\node at (-.5,-1.6) {$a$};\node at (.5,-1.5) {$b$};
\node at (-.5,-.6) {$a$};\node at (.5,-.5) {$d$};
\node at (-1.5,-1.5) {\tiniest{$r$}};\node at (-2,-.5) {\tiniest{$r+1$}};
\node at (-.6,-2.5) {\tiniest{$i$}};\node at (.6,-2.5) {\tiniest{$i+1$}};
\draw[->] (-.5,.1)--(-1.5,.9);
\draw[->] (.5,.1)--(1.5,.9);
\cell{2}{0}{}\cell{2}{1}{}\cell{1}{0}{}\cell{1}{1}{}
\node at (-.5-2,1.5) {$b$};\node at (.5-2,1.4) {$a$};
\node at (-.5-2,2.4) {$a$};\node at (.5-2,2.5) {$d$};
\node at (-1.5-2,1.5) {\tiniest{$r$}};\node at (-2-2,2.5) {\tiniest{$r+1$}};
\draw[blue] (-.9-2,1.9) rectangle (.9-2,1.1);
\node at (-.5+2,1.5) {$b$};\node at (.5+2,1.4) {$a$};
\node at (-.5+2,2.5) {$d$};\node at (.5+2,2.4) {$a$};
\draw[blue] (-.9+2,1.9) rectangle (.9+2,1.1);
\draw[blue] (-.9+2,2.9) rectangle (.9+2,2.1);
\end{tikzpicture}}

\end{center}

\begin{itemize}
\item[(i)] If $r=L$, $\widetilde{\tau}_i^{(r)}$ produces $\T_i^{(r)}(\sigma)$ with probability 1.
\item[(ii)] If $b=c$ (left) or $\{a,b\}\cap \{c,d\}=\emptyset$ and $\Qc(c,a,d)=\Qc(c,b,d)$ (right), $\widetilde{\tau}_i^{(r)}$ produces $\T_i^{(r)}(\sigma)$ with probability $1$.
\item[(iii)] If $b=d$ (left) or $\{a,b\}\cap \{c,d\}=\emptyset$ and $\Qc(c,a,d)\neq\Qc(c,b,d)$ (right), $\widetilde{\tau}_i^{(r)}$ produces the set $\widetilde{\tau}_i^{(r+1)}(\T_i^{(r)}(\sigma))$ with probability 1.
\item[(iv)] If $a=c$, $\widetilde{\tau}_i^{(r)}$ produces $\T_i^{(r)}(\sigma)$ with probability $(q^{\ell'}t^{A'})^{\Qc(a,b,d)}(1-t)/(1-q^{\ell'}t^{A'+1})$ and it produces the set $\widetilde{\tau}_i^{(r+1)}(\T_i^{(r)}(\sigma))$ with probability $t^{1-\Qc(a,b,d)}(1-q^{\ell'}t^{A'})/(1-q^{\ell'}t^{A'+1})$, where $A':=\arm(r+1,i+1)+1$ and $\ell':=\leg(r+1,i)+1=L-r$.
\end{itemize}
Suppose $\sigma'\in\widetilde{\tau}_i(\sigma)$ such that $\sigma'=\T_i^{[1,k]}(\sigma)$. Let $\prob_i^{(r)}(\sigma,\sigma')$ denote the probability that $\widetilde{\tau}^{(r)}_i$ applied to $\T_i^{[1,r-1]}(\sigma)$ produces $\widetilde{\tau}^{(r+1)}_i\left(\T_i^{[1,r]}(\sigma)\right)$ when $r<k$ and $\T_i^{[1,r]}(\sigma)$ when $r=k$. Then the probability that $\widetilde{\tau}_i$ applied to $\sigma$ produces $\sigma'$ is
\[
\prob_i(\sigma, \sigma')=\prod_{r=1}^{k} \prob_i^{(r)}(\sigma, \sigma').
\]
\end{defn}

\begin{defn}\label{def:dec lambda}
Let $\sigma\in\Tab'(\lambda)$ and set $k=\ell(\lambda)$. The \emph{bottom border} of $\sigma$, denoted $\Bot(\sigma)$, is a word in $\Z_{>0}^k$ given by the sequence of entries in the bottom row of $\sigma$, read from left to right:
\[
\Bot(\sigma)=\sigma(1,1)\sigma(1,2)\cdots\sigma(1,k).
\]
Define $\dec_{\lambda}(w)\in\Sym_{\lambda}(w)$ to be the unique arrangement of the bottom border such that the entries are decreasing within each block of columns of the same length: $\dec_{\lambda}(w)_i>\dec_{\lambda}(w)_{i+1}$ for all $\lambda$-compatible $i$. Then we define the \emph{length} of a bottom border $w$ with respect to $\lambda$ to be the number of $\lambda$-comparable coinversions: 
\[\ell'_{\lambda}(w):=\{i<j: \lambda_i=\lambda_j, w_i<w_j\}.
\]  
Since inversions and coinversions are equidistributed, we have $\sum_{v\in\Sym_{\lambda}(w)}t^{\ell'_{\lambda}(v)}=\perm_{\lambda}(t)$. Observe that $\sigma\in\Tab'(\lambda)$ with $\Bot(\sigma)=w$ is $\coinv$-sorted if and only if $w=\dec_{\lambda}(w)$.
\end{defn}

A similar case analysis to the one done in \cref{sec:operators} produces the $\inv$ analog of \cref{prop:balance}.
\begin{lemma}\label{lem:inv balance}
Let $\lambda$ be a partition and let $1\leq i\leq \ell(\lambda)$ be $\lambda$-compatible. Let $\sigma\in\Tab'(\lambda,n)$ with $\Bot(\sigma)=w$, such that $w_i>w_{i+1}$, 
and let $\sigma \in\widetilde{\tau}_i(\sigma)$. Then $\Bot(\sigma')=s_i\cdot w$ and
\begin{equation}\label{eq:prob balance inv}
\wt_{\HHL}(\sigma')\prob_i(\sigma',\sigma)=t\wt_{\HHL}(\sigma)\prob_i(\sigma,\sigma').
\end{equation}
\end{lemma}

Again, the same argument used to prove \cref{lem:perm} yields the $\inv$ analog. 
\begin{lemma}\label{lem:perm inv}
Let $\sigma\in\Tab'(\lambda)$ be an $\inv$-non-attacking tableau with $\Bot(\sigma)=w$. Then
\[
\wt_{\HHL}(\sigma)=t^{\ell_{\lambda}(w)}\sum_{\substack{\sigma'\in\Tab'(\lambda)\\\Bot(\sigma')=\dec_{\lambda}(w)}} \wt_{\HHL}(\sigma')\prob(\sigma',\sigma).
\]
Moreover, for any $\coinv$-sorted $\inv$-non-attacking tableau $\tau$ with bottom row $v$ such that $v=\dec_{\lambda}(v)$ and any $\lambda$-compatible permutation $w\in\Sym_{\lambda}(v)$,
\[
\sum_{\substack{\sigma'\in\Tab'(\lambda)\\\Bot(\sigma')=w}}\prob(\tau,\sigma')=1.
\]
\end{lemma}

\begin{proof}[Proof of \cref{thm:compact inv}] Let $\lambda$ be a partition and set $k:=\ell(\lambda)$. First, using \eqref{eq:PR perm}, we write
\[
\frac{\widetilde{\PR}_{\lambda}(q,t)}{\PR_{\lambda}(q,t)} = \perm_{\lambda}(q,t)^{-1}\Pi_{\lambda}(q,t).
\]
Then the right hand side of \eqref{eq:P inv} can be written as a sum over all possible bottom rows $w\in\Z^{k}_{>0}$ of $\inv$-non-attacking fillings in $\Tab'(\lambda)$: 
\begin{align}
P_{\lambda}(X;q,t)&= 
\perm_{\lambda}(t)^{-1}\Pi_{\lambda}(q,t)\sum_{w\in\Z_{>0}^{k}}\ \sum_{\substack{\sigma\in\Tab'(\lambda)\\\Bot(\sigma)=w}}\wt_{\HHL}(\sigma)\nonumber\\
&=\perm_{\lambda}(t)^{-1}\Pi_{\lambda}(q,t)\sum_{w\in\Z_{>0}^{k}}\ \sum_{\substack{\sigma\in\Tab'(\lambda)\\\Bot(\sigma)=w}}\ \sum_{\substack{\sigma'\in\Tab'(\lambda)\\\Bot(\sigma')=\dec_{\lambda}(w)}} t^{\ell_{\lambda}(w)}\wt_{\HHL}(\sigma')\prob(\sigma',\sigma)\label{eq:g1}\\
&=\Pi_{\lambda}(q,t) \sum_{\substack{\sigma\in\Tab'(\lambda)\\\sigma\,\cisort}} \wt_{\HHL}(\sigma) \label{eq:g3}
\end{align}
where \eqref{eq:g1} is by \cref{lem:perm inv}, and the final equality follows from \eqref{eq:perm} and the definition of a $\coinv$-sorted $\inv$-non-attacking filling. 
\end{proof}

\section{Comparison to Martin's multiline queues}\label{sec:martin}

In this section, we explain the map from sorted $\quinv$-non-attacking tableaux to Martin's multiline queues in \cite{martin-2020}, and show that at $x_1=x_2=\cdots=q=1$, the weights of our tableaux correspond to the weights of the multiline queues.

\begin{defn}[{\cite[Algorithms 1 and 2]{martin-2020}}]\label{def:mlq}
 Let $\lambda$ be a partition, fix an integer $n\geq\ell(\lambda)$, and set $L:=\lambda_1$. Define $\Z_n:=\Z/n\Z$ with representatives $\Z_n=\{1,\ldots,n\}$. A $n\times L$ \emph{cylindric lattice} is a lattice with $L$ rows numbered from bottom to top, and $n$ columns which are numbered by $\Z_n$ from left to right such that $n+1\equiv 1$.  A \emph{particle system} of type $(\lambda, n)$ is a configuration of particles on a $n\times L$ cylindric lattice, where each site of the lattice is occupied by at most one particle, and each row $j$ contains $\lambda_j$ particles. See \cref{fig:MLQ tab} for an example of a particle system (ignoring the lines linking the particles.) 
    
Fix a sequence of permutations $(\pi_{r,\ell})_{2\leq r\leq \ell \leq L}$ where $\pi_{r,\ell}\in S_{\lambda'_{\ell}-\lambda'_{\ell+1}}$ for $2\leq r\leq \ell \leq L$. A \emph{multiline queue} of type $(\lambda,n)$ is a particle system combined with a pairing of the particles between adjacent rows in an order determined by the sequence of permutations $(\pi_{r,\ell})_{r,\ell}$. Each pairing is assigned a weight in $\mathbb{Q}(t)$. The pairings and weights are obtained according to the following procedure. 
    \begin{itemize}
        \item For each row $r=L,L-1,\ldots,2$ in that order, begin by labeling all unlabeled particles with ``$r$''. 
        \item Within row $r$, let $\pi_{r,\ell}$ determine the pairing order of particles labeled $\ell$. For $\ell=L,\ldots,r$ in that order, consider each particle labeled $\ell$ in turn according to the order $\pi_{r,\ell}$. 
        \begin{itemize}
            \item[(i)] If there is an unlabeled particle in the same column in row $r-1$, pair to that particle. The weight of such a pairing is 1, and we call this a \emph{trivial pairing}.
            \item[(ii)] Otherwise, the pairing is \emph{non-trivial}. Choose an unlabeled particle to pair with in row $r-1$, and call this pairing $p$. The weight of the pairing $p$ is $t^{s(p)}(1-t)/(1-t^{f(p)+1})$ where $f(p)$ is the total number of unlabeled particles in row $r-1$ remaining after $p$ is made, and $s(p)$ is the number of unlabeled particles (cyclically) between the pair when looking to the right from the particle in row $r$. For such a pairing, denote $r(p):=r$ and $\ell(p):=\ell$.
            \item[(iii)] Label the particle in row $r-1$ that is paired to with the label ``$\ell$''.
        \end{itemize}
    \end{itemize}
    The weight $\wt(M)$ of a multiline queue $M$ is the product of the weights over all of its pairings:
    \[
    \wt(M)=\prod_{p} \frac{t^{s(p)}(1-t)}{1-t^{f(p)+1}}.
    \] 
    Denote the set of multiline queues of type $(\lambda,n)$ by $\MLQ_M(\lambda,n)$ (we use the notation $\MLQ_M$ to distinguish from the multiline queues defined in \cite{CMW18}).
\end{defn}

Let $\sigma\in\Tab(\lambda,n)$ be a $\coquinv$-sorted $\quinv$-non-attacking filling. We will map $\sigma$ to a multiline queue $M$ so that $\wt(\sigma)=\wt(M)$. First, for $1\leq r\leq \lambda_1$, the positions of the particles in row $r$ of the particle system of $M$ are given by the contents of row $r$ in $\sigma$: the columns in $M$ that contain particles are $\{\sigma(r,1),\ldots,\sigma(r,\lambda'_r)\}$. Next, the permutations $\pi_{r,\ell}$ dictating the pairing orders are obtained by standardizing the word obtained by scanning the entries of $\sigma$ in row $r$ in the columns of height $\ell$. Finally, the pairings are given as follows: a particle at site $(r,\sigma(r,k))$ in the multiline queue is paired to the particle at site $(r-1,\sigma(r-1,k))$. The $\quinv$-non-attacking condition on $\sigma$ implies condition (i) is satisfied, and the fact that $\sigma$ is sorted means each pairing of a given configuration is counted exactly once. See \cref{fig:MLQ tab}.

\begin{example}\label{fig:MLQ tab}
For $\sigma\in\Tab((4,4,3,3,1),9)$ below, with weight
\[
\wt(\sigma)=\frac{t^7q^6(1-t)^7}{(1-qt^4)(1-qt^3)(1-q^2t^4)(1-qt)(1-q^3t^4)(1-q^2t^3)(1-q^2t^2)},
\] 
we have the corresponding multiline queue $M$. 
\begin{center}

\begin{tikzpicture}[scale=0.7]
\node at (-2.5,0) {$\sigma=$};
\node at (0,0) {$\tableau{1&6\\6&8&2&5\\4&8&2&1\\4&2&7&4&9}$};

\def \w{.8};
\def \h{1};
\def \r{0.25};
    
\begin{scope}[xshift=6cm,yshift=-1.7cm]
\node at (-1,2) {$M=$};
\draw[gray!50,thin,xstep=\w,ystep=\h] (0,0) grid (9*\w,4*\h);

\foreach \xx in {1,...,9}
{
\node at (\w*\xx-0.5,-.2) {\tiny \xx};
}
\foreach \xx\yy\i in {0/3/4,5/3/4,5/2/4,7/2/4,4/2/3,1/2/3,3/1/4,7/1/4,1/1/3,0/1/3,3/0/4,1/0/4,6/0/3,4/0/3,8/0/1}
    {
    \draw (\w*.5+\w*\xx,\h*.5+\h*\yy) circle (\r cm);
    \node at (\w*.5+\w*\xx,\h*.5+\h*\yy) {\i};
    }

\draw[blue,thick] (\w*.5,\h*3.5-\r)--(\w*.5,\h*2.9)--(\w*5.5,\h*2.9)--(\w*5.5,\h*2.5+\r);
\draw[red,thick] (\w*5.5,\h*3.5-\r)--(\w*5.5,\h*3.1)--(\w*7.5,\h*3.1)--(\w*7.5,\h*2.5+\r);
    
\draw[blue,thick,-stealth] (\w*5.5,\h*2.5-\r)--(\w*5.5,\h*1.9)--(\w*9.3,\h*1.9);
\draw[blue,thick] (-.2,\h*1.9)--(\w*3.5,\h*1.9)--(\w*3.5,\h*1.5+\r);
\draw[red,thick] (\w*7.5,\h*2.5-\r)--(\w*7.5,\h*1.5+\r);
\draw[black!50!green,thick] (\w*1.5,\h*2.5-\r)--(\w*1.5,\h*1.5+\r);
\draw[orange,thick,-stealth] (\w*4.5,\h*2.5-\r)--(\w*4.5,\h*2.1)--(\w*9.3,\h*2.1);
\draw[orange,thick] (-.2,\h*2.1)--(\w*.5,\h*2.1)--(\w*.5,\h*1.5+\r);

\draw[blue,thick] (\w*3.5,\h*1.5-\r)--(\w*3.5,\h*.5+\r);
\draw[red,thick,-stealth] (\w*7.5,\h*1.5-\r)--(\w*7.5,\h*.9)--(\w*9.3,\h*.9);
\draw[red,thick] (-.2,\h*.9)--(\w*1.5,\h*.9)--(\w*1.5,\h*.5+\r);
\draw[black!50!green,thick] (\w*1.5,\h*1.5-\r)--(\w*1.5,\h*1.1)--(\w*6.5,\h*1.1)--(\w*6.5,\h*.5+\r);
\draw[orange,thick] (\w*.5,\h*1.5-\r)--(\w*.5,\h*1)--(\w*4.5,\h*1)--(\w*4.5,\h*.5+\r);

\end{scope}
\end{tikzpicture}
\end{center}
Matching the order of particles to their order in the tableau when read left to right, we have $\pi_{4,4}=\pi_{3,4}=\pi_{3,3}=\pi_{2,4}=12\in S_2$, $\pi_{2,3}=21\in S_2$. Thus the weights on the pairings for $(r,\ell)=(4,4)$ are $\frac{t^2(1-t)}{1-t^4}\cdot \frac{1-t}{1-t^3}$, for $(r,\ell)=(3,4)$ they are $\frac{t^3(1-t)}{1-t^4}\cdot 1$, for $(r,\ell)=(3,3)$ they are $1\cdot\frac{1-t}{1-t}$, for $(r,\ell)=(2,4)$ they are $1\cdot\frac{t(1-t)}{1-t^4}$, and for $(r,\ell)=(2,3)$ they are $\frac{t(1-t)}{1-t^3}\cdot \frac{1-t}{1-t^2}$. Thus 
\[\wt(M)=\frac{t^7(1-t)^7}{(1-t^4)^3(1-t^3)^2(1-t^2)(1-t)}=\wt(\sigma)\big\vert_{q=1}.\]

\end{example}

For each $r>1$ and $1\leq k\leq \lambda'_r$, the pairing $p_{r,j}$ in $M$ between the particles $(r,\sigma(r,k))$ and $(r-1,\sigma(r-1,k))$ is trivial if and only if $\sigma(r,k)=\sigma(r-1,k)$, and contributes 1 both to $\wt(M)$ and $\wt(\sigma)$.  Otherwise if the pairing $p$ is non-trivial, we observe that $f(p)$ (the number of remaining unlabeled particles in row $r-1$ after the pairing $p$ has been made) is equal to $\rarm(r,k)$. Similarly, $s(p)$ is the number of indices $k<j$ such that cyclically, $\sigma(r,k)<\sigma(r-1,j)<\sigma(r-1,k)$, which means $\Qc(\sigma(r,k),\sigma(r-1,k),\sigma(r-1,j))=1$ making that corresponding $L$-triple a coquinv triple. Thus $t^{s(p)}(1-t)/(1-t^{f(p)+1})$ is equal to the contribution to $\wt(\sigma)$ from this pair of cells. Taking the product over all pairings, we get the following lemma.

\begin{lemma}\label{cor:martin bijection}
Fix a partition $\lambda$ and an integer $n\geq \ell(\lambda)$. The multiline queues of type $(\lambda,n)$ defined in \cref{def:mlq} are in bijection with the set of $\coquinv$-sorted $\quinv$-non-attacking fillings in $\Tab(\lambda,n)$ with
\[
\sum_{M\in \MLQ_M(\lambda,n)}\wt(M)(t)=\sum_{\substack{\sigma\in\Tab(\lambda,n)\\\sigma\,\cqsort}}\wt(\sigma)(1,\ldots,1;1,t).
\]
\end{lemma}

Finally, if we associate to $M\in \MLQ_M(\lambda,n)$ a weight in $x_1,\ldots,x_n,q$, we obtain an alternative multiline queue formula for $P_{\lambda}(X;q,t)$ in terms of Martin's multiline queues.  Define $x^M:=\prod_i x_i^{m_i(M)}$ where $m_i(M)$ is the total number of particles in column $i$ of $M$, and define 
\[\maj(M):=\sum_p \ell(p)-r(p)+1
\] 
where the sum is over all pairings $p$ in $M$ that wrap around the cylinder. Define the weight $\wt(M)(X_n;q,t)$ to be
\[
\wt(M)(X_n;q,t):=x^Mq^{\maj(M)}\prod_p \frac{t^{s(p)}(1-t)}{1-q^{\ell(p)-r(p)+1}t^{f(p)+1}},
\]
where the product is over all non-trivial pairings $p$. (These definitions are identical to the corresponding definitions in \cite{CMW18}.) It is now straightforward to check that $\wt(\sigma)=\wt(M)$ where $\sigma\in\Tab(\lambda,n)$ is the unique $\coquinv$-sorted $\quinv$-non-attacking filling corresponding to $M$ according to the map described above.

\begin{prop}\label{cor:martin P}
The Macdonald polynomial is given by
\[
P_{\lambda}(x_1,\dots,x_n;q,t)=\sum_{M\in\MLQ_M(\lambda,n)}\wt(M)(X_n;q,t)
\]
where the sum is taken over Martin's multiline queues as defined in \cref{def:mlq}.
\end{prop}

\begin{remark}\label{rem:mlq}
    A subtle distinction exists between the multiline queues of \cite{martin-2020} described in \cref{def:mlq} and the multiline queues defined in \cite{CMW18}, which manifests when $\lambda$ is not a strict partition. The difference is in the following: in the multiline queues defined in \cite{CMW18}, for a given row $r$ and set of particles labeled $\ell$, all particles with an unlabeled particle directly below in row $r-1$ must pair first (necessarily to those particles directly below them). In particular, this disallows the pairings in which there are two particles with the same label at sites $(r,j)$ and $(r-1,j)$ that are not paired to each other. For example, the pairing in \cref{fig:MLQ tab} of particles labeled 4 from row 4 to row 3 is allowed in Martin's multiline queues, but disallowed according to \cite{CMW18}. Due to this technicality, it has been difficult to define ``natural'' statistics on tableaux to align the multiline queues from \cite{CMW18}, and there hasn't been a clear way to see the plethystic relationship between the corresponding multiline queue formula for $P_{\lambda}$ and the $\quinv$ tableaux formula for $\HH_{\lambda}$.
\end{remark}

Consequently, we conjecture an alternative tableaux formula for the \emph{ASEP polynomials} $f_{\alpha}(X_n;q,t)$, which form a nonsymmetric decomposition of $P_{\lambda}(X_n;q,t)$:
\[P_\lambda(X_n;q,t)=\sum_{\alpha\,:\,\sort(\alpha)=\lambda}f_{\alpha}(X_n;q,t).\]
Each polynomial $f_{\alpha}(X_n;q,t)$ coincides with the \emph{permuted basement Macdonald polynomial} \linebreak $E^{\tau}_{\inc(\alpha)}(X_n;q,t)$, where $\tau$ is any permutation such that $\alpha_{\tau(1)}\leq\alpha_{\tau(2)}\leq\cdots\leq\alpha_{\tau(n)}$, equivalently, $\tau\cdot \alpha:=(\alpha_{\tau_1},\ldots,\alpha_{\tau_n})=\inc(\alpha)$. See \cite{CMW18} for details on the ASEP polynomials, and \cite{Ale16,Fer11} for the general definition of $E_\gamma^\tau$.

For a (weak) composition $\alpha=(\alpha_1,\ldots,\alpha_n)$ with $k$ nonzero parts, define $\beta(\alpha)$ to be the set of words in $\Z_{>0}^k$ that select and reorder the positive parts of $\alpha$ into weakly decreasing order: 
\[\beta(\alpha)=\{\tau_1\cdots \tau_k\in\Z_{>0}^k\,:\  \alpha_{\tau_1}\geq \cdots \geq \alpha_{\tau_k}\}.\] 
Each word $\tau_1\cdots \tau_k\in\beta(\alpha)$ can be viewed as the first $k$ entries of a permutation $\tau=\tau_1\cdots\tau_n\in S_n$ such that $\tau\cdot\alpha=\sort(\alpha)$. We say a word in $\Z_{>0}^k$ is \emph{compatible with} $\alpha$ if it belongs to the set $\beta(\alpha)$. See \cref{ex:betaalpha}.

For a filling $\sigma$, we shall use the notation $r_1(\sigma)$ to denote the bottom row of $\sigma$.

\begin{conjecture}\label{conj:f}
Let $\alpha=(\alpha_1,\ldots,\alpha_n)$ be a composition with $\lambda=\sort(\alpha)$. The ASEP polynomial $f_{\alpha}(X_n;q,t)$ is given by
\[
f_{\alpha}(X_n;q,t)=\sum_{\substack{\sigma\in\Tab(\lambda,n)\\\sigma\,\cqsort\\r_1(\sigma)\in\beta(\alpha)}}\wt(\sigma)(X_n;q,t).
\]
where the sum is over all $\coquinv$-sorted non-attacking tableaux $\sigma$ whose bottom row $r_1(\sigma)$ is compatible with $\alpha$.
\end{conjecture}

\begin{example}\label{ex:betaalpha}
For instance if $\alpha=(0,2,3,1,0,2,1)$, then 
\[f_\alpha(X_6;q,t)=E_{(0,0,1,1,2,2,3)}^{\tau}(X_6;q,t),\qquad \tau=(1,5,4,7,2,6,3).
\]
The set of words corresponding to partial permutations compatible with $\alpha$ is 
\[\beta(\alpha)=\{3\,2\,6\,4\,7,\ 3\,6\,2\,4\,7,\ 3\,2\,6\,7\,4,\ 3\,6\,2\,7\,4\}.
\]
Then $f_\alpha(X_7;q,t)$ is the sum over all $\coquinv$-sorted non-attacking fillings 
$\sigma\in\Tab((3,2,2,1,1),7)$ with $r_1(\sigma)\in\beta(\alpha)$.
\end{example}

\begin{remark}
From the proof of \cref{cor:martin bijection}, we have that $f_{\alpha}(1,\ldots,1;1,t)$ is equal to the sum on the right hand side specialized at $X_n=q=1$. Equality can also be deduced at $t=0$ from comparing to the coinversion-free HHL tableaux formula for $E^{\sigma}_{\inc(\alpha)}(X_n;q,0)$. However, for the general case, one would need to show that the nonsymmetric components of the $\coquinv$ formula are a \emph{qKZ family} \cite[Definition 1.18]{CMW18}, possibly using similar techniques to the proofs in \cite{CMW18}. It would be interesting to find a combinatorial proof relating the formula of \cref{conj:f} to the HHL tableaux formula for $E_{\inc(\alpha)}^\sigma(X_n;q,t)$.
\end{remark}

\section{Jack specialization}\label{sec:jack}

The Jack polynomials $J_{\lambda}(X;\alpha)$ can be obtained as a limit of $(1-t)^{-|\lambda|}J_{\lambda}(X;t^{\alpha},t)$ with $t\rightarrow 1$. An elegant tableaux formula was found by Knop and Sahi in \cite[Theorem 4.10]{KS96}; this formula was recovered in \cite[Section 8]{HHL05} by specializing the HHL formula \eqref{eq:P inv} for $J_{\lambda}(X;q,t)$:
\begin{equation}\label{eq:Jack HHL}
J_{\lambda}(X;\alpha)=\sum_{\substack{\sigma\in\dg(\lambda)\rightarrow\Z^+\\\sigma\inona}}x^{\sigma}\prod_{\substack{u\in\overline{\dg}(\lambda)\\\sigma(u)=\sigma(\South(u))}} \alpha(\leg(u)+1)+\arm(u)+1.
\end{equation}
Knop and Sahi also found an alternative formula for $J_{\lambda}(X;\alpha)$ \cite{Spersonal}, although as far as the author is aware, this had not been published. We recover this formula by specializing \eqref{eq:J quinv} and taking the limit.
\begin{prop}
The Jack polynomial is given by
\begin{equation}\label{eq:Jack quinv}
J_{\lambda}(X;\alpha)=\sum_{\substack{\sigma\in\dg(\lambda)\rightarrow\Z^+\\\sigma\qnona}}x^{\sigma}\prod_{\substack{u\in\overline{\dg}(\lambda)\\\sigma(u)=\sigma(\South(u))}} \alpha(\leg(u)+1)+\rarm(u)+1.
\end{equation}
\end{prop}

One observes that just like \eqref{eq:P quinv} is a more efficient formula than \eqref{eq:P inv}, \eqref{eq:Jack quinv} is more efficient than \eqref{eq:Jack HHL} in having fewer terms due to the more restrictive $\quinv$-non-attacking condition. 

\begin{example} 
The following example demonstrates a computation of $[m_{2,1,1}]J_{3,1}(X;\alpha)$ using both formulas. Below we list all $\inv$-non-attacking tableaux and their weights for content $x_1^2x_2x_3$.
\begin{center}
\begin{tikzpicture}[scale=0.5]
\def\h{3cm};
\node[label=$1+2\alpha$] (A1) at (0*\h,0)  {\tableau{3\\1\\1&2}};
\node[label=$1+2\alpha$] (A2) at (1*\h,0) {\tableau{2\\1\\1&3}};
\node[label=$1+\alpha$]  (A3) at (2*\h,0) {\tableau{1\\1\\2&3}};
\node[label=$1+\alpha$]  (A4) at (3*\h,0) {\tableau{1\\1\\3&2}};
\node[label=$1$]  (A5) at (4*\h,0) {\tableau{1\\3\\1&2}};
\node[label=$1$]  (A6) at (5*\h,0) {\tableau{1\\2\\1&3}};
\node[label=$1$]  (A7) at (6*\h,0) {\tableau{1\\2\\3&1}};
\node[label=$1$]  (A8) at (7*\h,0) {\tableau{1\\3\\2&1}};
\node[label=$1$] (A9) at (8*\h,0)  {\tableau{3\\1\\2&1}};
\node[label=$1$] (A10) at (9*\h,0) {\tableau{2\\1\\3&1}};

\end{tikzpicture}
\end{center}
Now we list the $\quinv$-non-attacking tableaux and their weights for the same coefficient.
\begin{center}
\begin{tikzpicture}[scale=0.5]
\def\h{3cm};
\node[label=$2+2\alpha$] (A1) at (0*\h,0)  {\tableau{3\\1\\1&2}};
\node[label=$2+2\alpha$] (A2) at (1*\h,0) {\tableau{2\\1\\1&3}};
\node[label=$1+\alpha$]  (A3) at (2*\h,0) {\tableau{1\\1\\2&3}};
\node[label=$1+\alpha$]  (A4) at (3*\h,0) {\tableau{1\\1\\3&2}};
\node[label=$1$]  (A5) at (4*\h,0) {\tableau{1\\3\\1&2}};
\node[label=$1$]  (A6) at (5*\h,0) {\tableau{1\\2\\1&3}};
\node[label=$1$]  (A7) at (6*\h,0) {\tableau{1\\2\\3&1}};
\node[label=$1$]  (A8) at (7*\h,0) {\tableau{1\\3\\2&1}};
\end{tikzpicture}
\end{center}
Both formulas sum to $[m_{2,1,1}]J_{3,1}(X;\alpha)=10+6\alpha$.
\end{example}

We conclude with an open question. 
\begin{question}
Is there a statistical mechanics model whose stationary probabilities are proportional to nonsymmetric components of the Jack polynomial? Further, could such a model have dynamics connected to the $\quinv$ statistic on $\quinv$-non-attacking tableaux?
\end{question}

\medskip
\noindent\textbf{Acknowledgments.}
The author was supported by NSERC grant RGPIN-2021-02568. We thank Sylvie Corteel, Leonid Petrov, Siddhartha Sahi, Travis Scrimshaw, and Lauren Williams for valuable conversations. We also thank the anonymous referees for helpful comments on the first version of the paper.

\bibliographystyle{plain}
\bibliography{Macbib}

\end{document}